\documentclass[12pt]{amsart}
\usepackage{ben}

\title[Nonuniqueness NS]{On non-uniqueness of mild solutions and stationary singular solutions to the Navier-Stokes equations}
\author{Alexey Cheskidov}
\address{Institute for Theoretical Sciences, Westlake University, 600 Dunyu Road, 310030 Hangzhou, Zhejiang, China}
\email{cheskidov@westlake.edu.cn}

\author{Hedong Hou}
\address{Westlake Institute for Advanced Study, Westlake University, 600 Dunyu Road, 310030 Hangzhou, Zhejiang, China}
\email{houhedong@westlake.edu.cn}

\date{June 11, 2026}

\begin{document}

\begin{abstract}
    We prove that the unconditional uniqueness of mild solutions to the Navier-Stokes equations fails in \emph{all} the Besov spaces with negative regularity index, by constructing non-trivial stationary singular solutions via convex integration. We also establish uniqueness of stationary weak solutions in an endpoint critical space. Similar results are proved for the fractional Navier-Stokes equations with arbitrarily large power of the Laplacian in both Lebesgue and Besov spaces.
\end{abstract}

\maketitle

%%%%% Introduction
\section{Introduction}
\label{sec:intro}

Consider the Navier-Stokes equations
\begin{equation*}
    \label{e:NSE}
    \tag{NSE}
    \begin{cases}
        \partial_t u + \Div(u \otimes u) - \Delta u + \nabla p = 0 \\
        \Div u = 0 \\
        u(0)=u_{\initial}, \quad \Div u_{\initial} = 0,
    \end{cases}
\end{equation*}
on the $d$-dimensional torus $\bT^d = \bR^d/\bZ^d$. Let $0<T \le \infty$. For a divergence-free distribution $u_{\initial} \in \scrD'(\bT^d)$, we say $u:[0,T) \times \bT^d \to \bR^d$ is a \emph{mild} solution to \eqref{e:NSE} with initial data $u_{\initial}$ if it satisfies the integral equation
\[ u(t) = e^{t\Delta} u_{\initial} - B(u,u)(t) \]
in the sense of distributions $\scrD'((0,T) \times \bT^d)$, where the bilinear operator $B$ is formally given by the integral
\[ B(u,u)(t) := \int_0^t e^{(t-s)\Delta} \bP \Div(u \otimes u)(s) ds, \quad t \in (0,T). \]
Here, $\bP$ denotes the Leray projection on divergence-free vector fields, and the tensor product $u \otimes u$ is defined as a paraproduct by the Littlewood-Paley series, see Section \ref{ssec:intro-stationary-sing-sol} for a detailed discussion\footnote{We do not impose $u \in L^2_{\loc}$, so $u \otimes u$ may not belong to $L^1_{\loc}$. One may also refer to such mild solutions as \emph{singular mild solutions} to distinguish them from the classical mild solutions.}.

Based on the scaling property (when posed in $\bR^d$)
\[ u(t,x) \mapsto u_\lambda(t,x) := \lambda u(\lambda^2 t,\lambda x), \]
there is a chain of scale-invariant (also called \emph{critical}) spaces
\[ \DotH^{\frac{d}{2}-1} \hookrightarrow L^d \hookrightarrow \DotB^{-1+\frac{d}{q}}_{q,r} \hookrightarrow \bmo^{-1} \hookrightarrow \DotB^{-1}_{\infty,\infty}, \quad 2 \le d < q < \infty. \]

The pioneering work of Fujita and Kato \cite{Fujita-Kato1964-NS} (together with a complement of Chemin \cite{Chemin1992-Hn/2-1}) established local well-posedness of \eqref{e:NSE} in $\DotH^{\frac{d}{2}-1}$. Such results are also valid in $L^d$ due to Kato \cite{Kato1984-NS} and in $\DotB^{-1+\frac{d}{q}}_{q,r}$ due to Cannone \cite{Cannone1997-Besov} and Planchon \cite{Planchon1998-Besov}. In endpoint critical spaces, Koch and Tataru \cite{Koch-Tataru2001-BMO-1} first proved local well-posedness in $\vmo^{-1}$, the closure of Schwartz functions in $\bmo^{-1}$. On the other hand, Bourgain and Pavlovi\'c \cite{Bourgain-Pavlovic2008-NS-ill} proved ill-posedness of \eqref{e:NSE} in $\DotB^{-1}_{\infty,\infty}$. See also the work of the first author and Shvydkoy \cite{Cheskidov-Shvydkoy2010-NS-ill} for ill-posedness in $B^{-1}_{\infty,\infty}$, as well as the works of Yoneda \cite{Yoneda2010-NS-ill} and Wang \cite{Wang2015-NS-ill} in $\DotB^{-1}_{\infty,r}$. The reader can also refer to the works of Coiculescu and Palasek \cite{Coiculescu-Palasek2025-nonunique-BMO-1}, and the first author, Dai and Palasek \cite{Cheskidov-Dai-Palasek2025-NS} on non-uniqueness of mild solutions with initial data in $\bmo^{-1}$.

Meanwhile, Fabes, Jones and Rivière \cite{Fabes-Jones-Riviere1972-Lp} proved \emph{unconditional uniqueness} of mild solutions in subcritical Lebesgue spaces $L^p(\bR^d)$ for $d<p<\infty$. Namely, given the initial data in $L^p$, there exists at most one mild solution in the solution class $C([0,T];L^p)$. Furioli, Lemarié-Rieusset and Terraneo \cite{Furioli-Lemarie-Rieusset-Terraneo2000-C0L3} extended it to the critical Lebesgue space $L^d(\bR^d)$ for $d \ge 3$ (see also \cite{Monniaux1999-C0L3,Lions-Masmoudi2001-C0L3}), while the case $L^2(\bR^2)$ remains open.

Recently, Buckmaster and Vicol \cite{Buckmaster-Vicol2019-C0L2} disproved unconditional uniqueness in the \emph{supercritical} space $H^\beta(\bT^3)$ for some $\beta>0$, using convex integration. See also the work of Luo \cite{Luo2019-NS-stationary} for $d \ge 4$. Tao \cite{Tao2019} extended it to $H^\beta(\bT^{d(\beta)})$ for any $\beta<1/2$, with the dimension $d(\beta)$ depending on $\beta$.

In this paper, we show that unconditional uniqueness of the Navier-Stokes equations \eqref{e:NSE} \emph{fails} for \emph{all} Besov spaces with negative regularity index, particularly including those in the (sub-)critical regime. To the best of our knowledge, this is the first non-uniqueness result in subcritical solution classes.

\subsection{Failure of unconditional uniqueness}
\label{ssec:intro-uu}
Let us briefly recall the definition of Besov spaces. The reader can refer to \cite{Triebel1983,Bahouri-Chemin-Danchin2011-book,Sawano2018-book-Besov} for more details of the function spaces to be used below. 

For any distribution $f \in \scrD'(\bT^d)$, denote by $\hatf$ the Fourier transform of $f$ in $\scrS'(\bZ^d)$, given by the formula
\[ \hatf(m) = \langle f, e^{-2\pi im \cdot x} \rangle, \quad m \in \bZ^d, \]
where $\langle\cdot,\cdot\rangle$ denotes the dual pairing between distributions and smooth functions on $\bT^d$. Let $\chi \in \Cc(\bR^d)$ be a radial bump function that is supported in $B(0,1)$ and equals 1 on $B(0,9/10)$. For any dyadic integer $N \ge 1$, define the \emph{Littlewood-Paley projection $P_{\le N}$} by
\[ P_{\le N} f(x) := \sum_{m \in \bZ^d} \chi(m/N) \hatf(m) e^{2\pi im \cdot x}, \]
and the companion projections
\[ P_{>N} := \id - P_{\le N}, \quad P_1 := P_{\le 1}, \quad P_N := P_{\le N} - P_{\le N/2} ~ \text{for } N \ge 2. \]
Let $s \in \bR$ and $q,r \in [1,\infty]$. The \emph{Besov space} $B^s_{q,r}$ consists of distributions $f \in \scrD'(\bT^d)$ for which
\[ \|f\|_{B^s_{q,r}} := \left( \sum_{N} N^{sr} \|P_N f\|_{L^q}^r \right)^{1/r} < \infty, \]
with the canonical modification when $r=\infty$. Define
\[
\bP B^s_{q,r} := \left\{ f \in B^s_{q,r} : \Div f = 0, ~ \hatf(0)=0 \right\}. 
\]
The same notation also applies to the Lebesgue spaces $L^p(\bT^d)$.

Our first theorem asserts that unconditional uniqueness of \eqref{e:NSE} fails in $B^{-\theta}_{q,r}(\bT^d)$ for any $\theta>0$ and $d \ge 2$.
\begin{theorem}[Failure of unconditional uniqueness of \eqref{e:NSE}]
    \label{thm:nonunique-mild-NSE}
    Let $d \ge 2$, $\theta>0$, and $q,r \in [1,\infty]$. The set of initial data $u_{\initial} \in \bP B^{-\theta}_{q,r}(\bT^d)$, which admit two distinct mild solutions
    \[ u,v \in C([0,T];B^{-\theta}_{q,r}) \quad \text{ for some } T>0 \]
    to the Navier-Stokes equations \eqref{e:NSE}, is dense in $\bP B^{-\theta}_{q,r}(\bT^d)$ (or weak*-dense if $r=\infty$).
\end{theorem}

In fact, for such initial data, we shall construct two mild solutions that are distinct from each other at any positive time, i.e.,
\[ u(0)=v(0)=u_{\initial}, \text{ but } u(t) \ne v(t) \quad \text{ for all } t \in (0,T]. \]

Combining this result with \cite{Cannone1997-Besov,Planchon1998-Besov} and the work of Sawada \cite{Sawada2003-NS-subcritical-Besov} on local well-posedness in subcritical Besov spaces, we conclude that the Cauchy problem of the Navier-Stokes equations \eqref{e:NSE} is locally well-posed but not unconditionally well-posed in (sub-)critical Besov spaces with negative regularity index.

The key mechanism behind our non-uniqueness results is that many admissible initial data are themselves stationary singular solutions of the Navier-Stokes equations (cf. Theorem \ref{thm:main}). Such a datum generates a time-independent mild solution $u(t)\equiv u_{\mathrm{in}}$, while the classical fixed-point argument yields another mild solution with the same initial data. Thus, non-uniqueness arises from the coexistence of a stationary singular flow and the classical local mild flow. Closely related examples were recently obtained by Lemari\'e-Rieusset \cite{Lemarie-Rieusset2025-sing-sol},  Ashkarian, Bhargava, Gismondi and Novack \cite{Ashkarian-Bhargava-Gismondi-Novack2025}, and Fujii \cite{Fujii2026-nonuniqueness-NS}. The novelty here is that this mechanism is turned into a systematic construction in all negative Besov spaces, including subcritical ones.

Moreover, applying the classical stability argument (see e.g., \cite[Proposition 4.1]{Gallagher-Iftimie-Planchon2003-stability}) to these stationary singular solutions yields that the unconditional uniqueness fails for \emph{all} initial data in subcritical and non-endpoint critical Besov spaces.

\begin{cor}[Strong failure of unconditional uniqueness of \eqref{e:NSE} in the (sub-)critical regime]
    \label{cor:UnU-subcritical}
    Let $d \ge 2$. Let $\theta$, $q$, and $r$ be such that
    \begin{enumerate}[label=\normalfont(\roman*)]
        \item \label{item:UnU-subcritical-nonend}
        either $d < q < \infty$, $1 \le r < \infty$, and $0<\theta \le 1-\frac{d}{q}$;
        \item \label{item:UnU-subcritical-end}
        or $q=\infty$, $1 \le r < \infty$, and $0<\theta<1$.
    \end{enumerate}
    Then for any initial data $u_{\initial} \in \bP B^{-\theta}_{q,r}(\bT^d)$, there exist two distinct mild solutions
    \[ u,v \in C([0,T];B^{-\theta}_{q,r}) \quad \text{ for some } T>0 \]
    to the Navier-Stokes equations \eqref{e:NSE} such that
    \[ u(0)=v(0)=u_{\initial}, \text{ but } u(t) \ne v(t) \quad \text{ for all } t \in (0,T]. \]
\end{cor}

\begin{remark}
    Our result indicates sharpness of the existing results on uniqueness of mild solutions (see, e.g., \cite{Cannone1997-Besov,Planchon1998-Besov,Gallagher-Koch-Planchon2016-blowup-criterion-Besov}) in the critical Chemin-Lerner space $\tilL^p_t B^{-1+\frac{d}{q}+\frac{2}{p}}_{q,r}$ for
    \[ d<q<\infty, \quad 1 \le r< \infty, \quad 2<p<\frac{2q}{q-d}. \]
    By contrast, our non-trivial stationary solutions belong to this Chemin-Lerner space when $\frac{2q}{q-d}<p \le \infty$. See the end of Section \ref{sec:pf-prop} for a more detailed discussion.
\end{remark}

\subsection{Fractional Navier-Stokes equations}
\label{ssec:intro-NSEa}
Similar results also hold for the fractional Navier-Stokes equations
\begin{equation*}
    \label{e:NSEa}
    \tag{NSE${}^{\alpha}$}
    \begin{cases}
        \partial_t u + \Div(u \otimes u) + (-\Delta)^\alpha u + \nabla p = 0 \\
        \Div u = 0 \\
        u(0)=u_{\initial}, \quad \Div u_{\initial} = 0,
    \end{cases}
\end{equation*}
where $(-\Delta)^\alpha$ denotes the fractional Laplacian operator defined on distributions by the formula
\[ (-\Delta)^\alpha u(x) := \sum_{m \in \bZ^d} (2\pi|m|)^{2\alpha} \hatu(m) e^{2\pi i m \cdot x}, \quad u \in \scrD'(\bT^d). \]
The series converges in the sense of distributions. For a divergence-free distribution $u_{\initial} \in \scrD'(\bT^d)$, we say $u:[0,T) \times \bT^d \to \bR^d$ is a \emph{mild} solution to \eqref{e:NSEa} with initial data $u_{\initial}$ if it satisfies the integral equation
\[ u(t) = e^{-t(-\Delta)^{\alpha}} u_{\initial} - B_\alpha(u,u)(t) \]
in $\scrD'((0,T) \times \bT^d)$. The bilinear operator $B_\alpha$ is formally given by the integral
\begin{equation}
    \label{e:Ba}
    B_\alpha(u,u)(t) := \int_0^t e^{-(t-s)(-\Delta)^\alpha} \bP \Div(u \otimes u)(s) ds, \quad t \in (0,T).
\end{equation}

The scaling property of \eqref{e:NSEa} (posed in $\bR^d$) reads as
\[ u(t,x) \mapsto u_\lambda(t,x) := \lambda^{2\alpha-1} u(\lambda^{2\alpha} t,\lambda x). \]
The (homogeneous) Besov space $\DotB^{-\theta}_{q,r}(\bR^d)$ is critical if
\[ \theta = 2\alpha-1-\frac{d}{q}. \]
When $\alpha>1/2$, there are (sub-)critical Besov spaces with negative regularity index. Our next theorem shows that unconditional uniqueness of \eqref{e:NSEa} fails in such spaces.
\begin{theorem}[Failure of unconditional uniqueness of \eqref{e:NSEa}]
    \label{thm:nonunique-mild-NSEa}
    Let $d \ge 2$ and $\alpha>1/2$. Let $\theta>0$ and $q,r \in [1,\infty]$.
    \begin{enumerate}[label=\normalfont(\roman*)]
        \item \label{item:nonunique-NSEa-B}
        The set of initial data $u_{\initial} \in \bP B^{-\theta}_{q,r}(\bT^d)$, which admit two distinct mild solutions
        \[ u,v \in C([0,T];B^{-\theta}_{q,r}) \quad \text{ for some } T>0 \]
        to the fractional Navier-Stokes equations \eqref{e:NSEa}, is dense in $\bP B^{-\theta}_{q,r}(\bT^d)$ (or weak*-dense if $r=\infty$).

        \item \label{item:nonunique-NSEa-Lp}
        Moreover, if $\alpha>(d+2)/4$, then the same statements also hold in $L^p(\bT^d)$ for any $1 \le p < 2$.
    \end{enumerate}
\end{theorem}
By scaling, when $\alpha>(d+2)/4$, $L^2(\bR^d)$ becomes subcritical. In this case, Lions \cite{Lions1969-book} proved global regularity of \eqref{e:NSEa}, see also \cite{Mattingly-Sinai1999-5/4-NS,Katz-Pavlovic2002-NSEa-hyper-reg,Tao2009-5/4-NS}. 

However, there are two conflicting scenarios for local well-posedness of \eqref{e:NSEa} in Besov spaces with negative regularity index. When $1/2<\alpha \le 1$, Wu \cite{Wu2006-NS-hypo-lwp} proved local well-posedness in the critical Besov spaces $\DotB^{1-2\alpha+\frac{d}{q}}_{q,r}$ for $2<q<\infty$. But for $\alpha \ge 1$, \eqref{e:NSEa} is \emph{ill-posed} in the \emph{subcritical} Besov spaces $\DotB^{-\alpha}_{\infty,r}$, as proved by Bourgain and Pavlovi\'c \cite{Bourgain-Pavlovic2008-NS-ill} for $\alpha=1$ and $r=\infty$, and later extended by the first author and Dai \cite{Cheskidov-Dai2014-NSEa} to all $\alpha \ge 1$ and $2<r\le \infty$. This suggests that
\begin{equation}
    \label{e:theta_alpha-infty}
    \theta_{\alpha,\infty} := \sup\{\theta>0: \eqref{e:NSEa}\text{ is locally well-posed in } B^{-\theta}_{\infty,\infty}\} \le \alpha.
\end{equation}

In Proposition \ref{prop:lwp-NSEa}, we present a result on local well-posedness for all $\alpha>1/2$. In particular, it provides a lower bound on $\theta_{\alpha,\infty}$ as
\[ \theta_{\alpha,\infty} \ge \alpha-1/2, \]
which indicates sharpness of \eqref{e:theta_alpha-infty} with respect to the order of $\alpha$ as $\alpha \to \infty$.

\begin{remark}
    The statements in Corollary \ref{cor:UnU-subcritical} also apply to the fractional Navier-Stokes equations \eqref{e:NSEa}. A detailed formulation is left to the reader.
\end{remark}

\subsection{Stationary singular solutions}
\label{ssec:intro-stationary-sing-sol}
Our main strategy to prove Theorems \ref{thm:nonunique-mild-NSE} and \ref{thm:nonunique-mild-NSEa} is to construct non-trivial \emph{singular} solutions to the stationary fractional Navier-Stokes equations
\begin{equation}
    \label{e:NS}
    \tag{NS${}^\alpha$}
    \begin{cases}
        \Div(u \otimes u) + (-\Delta)^\alpha u + \nabla p = 0, \quad x \in \bT^d \\
        \Div u = 0 \\
        \hatu(0) = 0.
    \end{cases}
\end{equation}
By singular solutions, we mean solutions that are not $L^2$-integrable, for which the tensor product $u \otimes u$ does not belong to $L^1$. Instead, we define it as a paraproduct in the space of \emph{distributions modulo constants} $\scrD'/\bC$, which identifies with the dual of the space of smooth functions of mean zero. 

More precisely, we define the tensor product $u \otimes u$ in homogeneous Sobolev spaces with negative regularity index. Recall that for $s \in \bR$, the \emph{homogeneous Sobolev space} $\DotH^s$ consists of $f \in \scrD'/\bC$ for which
\[ \|f\|_{\DotH^s} := \left( \sum_{m \in \bZ^d \setminus \{0\}} |m|^{2s} |\hatf(m)|^2 \right)^{1/2} < \infty. \]
Let $u \in \scrD'(\bT^d)$ and $s>0$. Then $u \otimes u$ is said to be \emph{well-defined as a paraproduct in $\DotH^{-s}$}, if
\begin{equation}
    \label{e:u*u-condition}
    \sum_{N,M} \left\| P_N u \otimes P_M u \right\|_{\DotH^{-s}} < \infty,
\end{equation}
where $N$ and $M$ range over all dyadic integers. Define the paraproduct 
\begin{equation}
    \label{e:u*u}
    u \otimes u := \sum_{N,M} P_N u \otimes P_M u \quad \text{ in } \DotH^{-s}.
\end{equation}
Since the series in \eqref{e:u*u} converges absolutely in $\DotH^{-s}$, we have
\begin{equation}
    \label{e:u*u-limit}
    u \otimes u = \lim_{N \to \infty} P_{\le N} u \otimes P_{\le N} u \quad \text{ in } \DotH^{-s}.
\end{equation}
In fact, under the condition \eqref{e:u*u-condition}, one may also use Bony's paraproduct decomposition to define $u \otimes u$, see, e.g., \cite{Bony1981-paraproduct,Coifman-Meyer1978,Bahouri-Chemin-Danchin2011-book}. 

Equivalently, one may view $u \otimes u$ as a \emph{renormalized product}. If $u \otimes u$ is well-defined as a paraproduct in $\DotH^{-s}$, the nonzero Fourier modes of $P_{\le N}u \otimes P_{\le N}u$ converge as $N \to \infty$, whereas the zeroth mode may diverge. Since the zeroth Fourier mode is annihilated by $\Div$, the renormalization consists in subtracting this mode, namely the spatial average $\fint P_{\le N} u \otimes P_{\le N} u$, and then passing to the limit. In this way, the product is naturally defined modulo constants.

We also identify $u \otimes u$ with a distribution in the inhomogeneous Sobolev space $H^{-s}$, via the canonical embedding $\iota$ given by the formula
\[ \iota(f) := \sum_{m \in \bZ^d \setminus \{0\}} \hatf(m) e^{2\pi i m \cdot x}, \quad f \in \DotH^{-s}. \]
Recall that the \emph{inhomogeneous Sobolev space} $H^{-s}$ consists of distributions $f \in \scrD'(\bT^d)$ for which
\[ \|f\|_{H^{-s}} := \left( \sum_{m \in \bZ^d} \langle m \rangle^{-2s} |\hatf(m)|^2 \right)^{1/2} < \infty, \]
where $\langle m \rangle$ denotes the Japanese bracket $\langle m \rangle := (1+|m|^2)^{1/2}$. When $u \in L^2(\bT^d)$, one has
\[ \iota(u \otimes u) = u \otimes u - \fint u \otimes u. \]
For convenience, we also abbreviate $\iota(u \otimes u)$ as $u \otimes u$ in $H^{-s}$.

\begin{definition}[Singular solutions]
    \label{def:sing-sol}
    We say $u \in \scrD'(\bT^d)$ is a \emph{singular solution to \eqref{e:NS}} if $\hatu(0)=0$; $u$ is weakly divergence free; there exists $s>0$ so that $u \otimes u$ is well-defined as a paraproduct in $\DotH^{-s}$; and for any divergence-free $\varphi \in C^\infty(\bT^d)$, it holds that
    \begin{equation}
        \label{e:sing-sol-id}
        -\left\langle u \otimes u, \nabla \varphi \right\rangle + \langle u,(-\Delta)^\alpha \varphi \rangle = 0,
    \end{equation}
    where $\langle\cdot,\cdot\rangle$ denotes the dual pairing between
    distributions and smooth functions on $\bT^d$.
\end{definition}

Note that any stationary weak solution $u \in L^2(\bT^d)$ to \eqref{e:NS} is indeed a singular solution, and every stationary singular solution that belongs to $L^2(\bT^d)$ is a weak solution. It is known that the existence of non-trivial stationary weak solutions to \eqref{e:NS} indicates non-uniqueness of weak solutions to \eqref{e:NSEa}. For $\alpha=1$, Luo \cite{Luo2019-NS-stationary} proved existence of non-trivial stationary weak solutions to \eqref{e:NS} for $d \ge 4$. 

The notion of singular solutions allows us to study stationary solutions with infinite energy. To the best of our knowledge, such singular solutions defined by renormalized products were first considered by Christ in the context of evolutionary 1D nonlinear Schr\"odinger equations \cite{Christ2005-NLS-singular} and 2D Navier-Stokes equations \cite{Christ2005-NS-singular}. The stationary counterpart was recently considered by Lemarié-Rieusset \cite{Lemarie-Rieusset2025-sing-sol}. Our following proposition shows that they correspond to mild solutions to \eqref{e:NSEa}.
\begin{prop}
    \label{prop:sing->mild}
    Let $u \in \scrD'(\bT^d)$ be a singular solution to \eqref{e:NS}. Then
    \[ U := \I_{(0,\infty)}(t) \otimes u(x) \quad \text{ in } \scrD'((0,\infty) \times \bT^d) \]
    is a mild solution to \eqref{e:NSEa} with initial data $U(0)=u$.
\end{prop}

The loss of unconditional uniqueness of mild solutions to the (fractional) Navier-Stokes equations, cf. Theorems \ref{thm:nonunique-mild-NSE} and \ref{thm:nonunique-mild-NSEa}, hence follows from existence of non-trivial stationary singular solutions.

\begin{theorem}[Density of non-trivial stationary singular solutions]
    \label{thm:main}
    Let $d \ge 2$ and $\alpha>0$. Let $1 \le p < 2$, $\theta>0$, and $q,r \in [1,\infty]$. The set of singular solutions to the stationary Navier-Stokes equations \eqref{e:NS}
    \[ u \in \bigcap_{1 \le \tilde p < 2} L^{\tilde p}(\bT^d) \cap \bigcap_{\substack{\tilde \theta>0\\ \tilde q, \tilde r \in [1,\infty]}} B^{-\tilde \theta}_{\tilde q,\tilde r} (\bT^d) \]
    is dense in $\bP L^p(\bT^d)$ and $\bP B^{-\theta}_{q,r}(\bT^d)$ (or weak*-dense if $r=\infty$).
\end{theorem}

For $\alpha=1$, Lemari\'e-Rieusset \cite{Lemarie-Rieusset2025-sing-sol} first proved existence of non-trivial singular solutions to \eqref{e:NS} in $H^{-1} \cap \bmo^{-1}(\bT^2)$. Later, Ashkarian, Bhargava, Gismondi and Novack \cite{Ashkarian-Bhargava-Gismondi-Novack2025} extended it to $\bigcap_{\epsilon \in (0,1)} L^{2-\epsilon} \cap \DotH^{-\epsilon}(\bT^2)$. We demonstrate that convex integration can be used to construct non-trivial solutions in every Besov space with negative regularity index, and, more importantly, for an arbitrarily large power of the Laplacian. 

\subsection{Weak solutions}
\label{ssec:intro-weak}
As aforementioned, a mild solution (resp. stationary singular solution) $u$ is a weak solution if $u \in L^2_{\loc}((0,\infty) \times \bT^d)$ (resp. $u \in L^2(\bT^d)$). The following theorem is concerned with non-uniqueness of (stationary) weak solutions, which generalizes the result in \cite{Luo2019-NS-stationary} for $\alpha=1$ and $d \ge 4$.
\begin{theorem}[Non-uniqueness of weak solutions]
    \label{thm:weak}
    Let $d \ge 2$ and $0<\alpha<\frac{d+1}{4}$. For any $\epsilon>0$, there exists a non-trivial weak solution $u \in L^2(\bT^d)$ to the stationary Navier-Stokes equations \eqref{e:NS} such that
    \begin{equation}
        \label{e:weak-L2-bd}
        \|u\|_{L^2} < \epsilon.
    \end{equation}
    As a consequence, there exists initial data $u_{\initial} \in L^2(\bT^d)$ with $\|u_{\initial}\|_{L^2}<\epsilon$, which admits two distinct weak solutions
    \[ u,v \in L^\infty([0,\infty);L^2) \cap C_{\ssfw}([0,\infty);L^2) \]
    to the fractional Navier-Stokes equations \eqref{e:NSEa} such that
    \[ u(0)=v(0)=u_{\initial}, \text{ but } u(t) \ne v(t) \quad \text{ for all } t>0. \]
    Here, $ C_{\ssfw}([0,\infty);L^2)$ is the space of weak*-continuous functions valued in $L^2(\bT^d)$.
\end{theorem}

We also show that there does not exist non-trivial stationary weak solutions to \eqref{e:NSE} in an endpoint critical space $B^{-1}_{\infty,1}(\bT^d)$. This improves a result by Lemarié-Rieusset \cite{Lemarie-Rieusset2025-sing-sol} for $d=2$ in $L^{2,1}(\bT^2)$, since $L^{2,1}(\bT^2)$ embeds into $L^2 \cap B^{-1}_{\infty,1}(\bT^2)$.

\begin{theorem}[Uniqueness of stationary solutions]
    \label{thm:unique-NS-L2B-1}
    There does not exist a non-trivial stationary solution $u$ to \eqref{e:NS} satisfying
    \[ \begin{cases} 
    u \in L^2 \cap B^{1-2\alpha}_{\infty,1}(\bT^d) & \text{ if } 0<\alpha \le 1 \\
    u \in L^2 \cap B^{-1}_{q,1}(\bT^d) & \text{ if } 1 \le \alpha \le \frac{d+2}{4} \\
    u \in L^2(\bT^d) & \text{ if } \alpha>\frac{d+2}{4}
    \end{cases}, \]
    where $q=\frac{d}{2\alpha-2}$.
\end{theorem}
We emphasize that the Besov spaces imposed in Theorem \ref{thm:unique-NS-L2B-1} always lie on the critical line.

After the completion of the present paper, we became aware of recent work by Fujii \cite{Fujii2026-nonuniqueness-NS} who constructed nontrivial singular stationary solutions in \emph{critical} Besov spaces with negative smoothness. This also implies non-uniqueness of mild solutions in those spaces. We emphasize that our singular stationary solutions belong to all the Besov spaces with negative smoothness, not only critical ones.

\subsection{Notation}
\label{ssec:notation}
Throughout the paper, without special mention, $N$ and $M$ are dyadic integers, by which we mean integers that are powers of 2. 

We say $X \lesssim Y$ (or $X \lesssim_A Y$, resp.) if $X \le CY$ with an irrelevant constant $C$ (or depending on $A$, resp.).

For any vectors $u,v \in \bR^d$, the tensor product $u \otimes v$ is defined by $(u \otimes v)_{i,j} = u_i v_j$. For any matrix-valued function $A:\bT^d \to \bR^{d \times d}$, the divergence operator is defined by $(\Div A)_i = \partial_j A_{i,j}$.

The pair $\langle \cdot,\cdot \rangle$ denotes the dual pairing between distributions and test functions on the prescribed set, which is clear from the context.

\subsection{Organization}
\label{ssec:organization}
Section \ref{sec:lwp} is concerned with local well-posedness (cf. Proposition \ref{prop:lwp-NSEa}) and unconditional uniqueness (cf. Corollary \ref{cor:uu-NSEa}) of the fractional Navier-Stokes equations \eqref{e:NSEa} in subcritical Lebesgue and Besov spaces.

Section \ref{sec:sing->mild} presents the proof of Proposition \ref{prop:sing->mild} and uniqueness of stationary solutions in subcritical Lebesgue spaces (cf. Corollary \ref{cor:sing-sol-Lp=0}). 

Section \ref{sec:iteration} collects two main propositions (cf. Propositions \ref{prop:iteration} and \ref{prop:L2-iteration}) on the iterative procedure. Admitting these two propositions, we prove the main theorems in Section \ref{sec:pf-thms}.

Sections \ref{sec:convex-integration}, \ref{sec:est}, and \ref{sec:pf-prop} are devoted to the proof of Proposition \ref{prop:iteration} via convex integration. Section \ref{sec:pf-L2-iteration} presents the proof of Proposition \ref{prop:L2-iteration}.
%%%%% LWP
\section{Local well-posedness of the fractional Navier-Stokes equations}
\label{sec:lwp}

Let us start with local well-posedness of \eqref{e:NSEa} in subcritical spaces.

\begin{prop}[Local well-posedness of \eqref{e:NSEa}]
    \label{prop:lwp-NSEa}
    Let $d \ge 2$ and $\alpha>1/2$. Let $\max\left\{\frac{d}{2\alpha-1}, 1\right\} < p < \infty$, $\max\left\{\frac{d}{2\alpha-1}, 1\right\} < q \le \infty$, $0<\theta<\frac{1}{2}(2\alpha-1-\frac{d}{q})$, and $1 \le r < \infty$.
    \begin{enumerate}[label=\normalfont(\roman*)]
        \item \label{item:lwp-NSEa-Lp}
        For every $R>0$ and every $\max\{p,p'\}\le \rho \le \infty$, there exists a time
        $T=T(R,\rho)>0$ such that every divergence-free $u_{\initial}\in L^p(\bT^d)$ with
        $\|u_{\initial}\|_{L^p}\le R$ gives rise to a unique mild solution $u$ to
        \eqref{e:NSEa} in the solution class
        \[
        X_T := \left\{ u \in C([0,T];L^p):t^{\frac{d}{2\alpha}(\frac{1}{p}-\frac{1}{\rho})}
        u(t) \in L^\infty((0,T);L^\rho) \right\}.
        \]
        The solution map is continuous from
        \[
        \{u_{\initial}\in L^p(\bT^d): \Div u_{\initial}=0,\ \|u_{\initial}\|_{L^p}\le R\}
        \]
        to $C([0,T];L^p(\bT^d))$.

        \item \label{item:lwp-NSEa-B}
        For every $R>0$ and every $\theta+\frac{d}{2q}<\beta<\frac12(2\alpha-1)$, there exists a time
        $T=T(R,\beta)>0$ such that every divergence-free
        $u_{\initial}\in B^{-\theta}_{q,r}(\bT^d)$ with
        $\|u_{\initial}\|_{B^{-\theta}_{q,r}}\le R$ gives rise to a unique mild solution
        $u$ to \eqref{e:NSEa} with initial data $u_{\initial}$ in the solution class
        \[
        Y_T := \left\{ u \in C([0,T];B^{-\theta}_{q,r}):
        t^{\frac{\beta}{\alpha}} u(t) \in L^\infty((0,T);B^{-\theta+2\beta}_{q,r}) \right\}.
        \]
        The solution map is continuous from
        \[
        \{u_{\initial}\in B^{-\theta}_{q,r}(\bT^d): \Div u_{\initial}=0,\
        \|u_{\initial}\|_{B^{-\theta}_{q,r}}\le R\}
        \]
        to $C([0,T];B^{-\theta}_{q,r}(\bT^d))$.
    \end{enumerate}
    As a consequence, \eqref{e:NSEa} is locally well-posed in $L^p(\bT^d)$ and $B^{-\theta}_{q,r}(\bT^d)$. Moreover, in both cases, $u \in C^\infty((0,T) \times \bT^d)$.
\end{prop}

\begin{remark}[Local well-posedness in endpoint spaces]
The proof below also indicates the following results in endpoint spaces.
    \begin{enumerate}
        \item For $p=\infty$, \eqref{e:NSEa} is locally well-posed in $L^\infty$, equipped with the weak*-topology induced by $L^1$. The solution class is $C((0,T];L^\infty) \cap C_\ssfw([0,T];L^\infty)$.

        \item For $r=\infty$, \eqref{e:NSEa} is locally well-posed in $B^{-\theta}_{q,\infty}$, equipped with the weak*-topology induced by $B^{\theta}_{q',1}$. The solution class is $\{u \in C_\ssfw([0,T];B^{-\theta}_{q,\infty}): t^{\beta/\alpha} u(t) \in L^\infty((0,T);B^{-\theta+2\beta}_{q,\infty}\}$.
    \end{enumerate}
Indeed, the heat semigroup is not necessarily strongly continuous but merely bounded and weak*-continuous in these endpoint spaces.
\end{remark}

\begin{remark}
    The statements in Proposition \ref{prop:lwp-NSEa} also hold in $\bR^d$.
\end{remark}

The proof is based on the classical smoothing effects of the heat kernel and the Oseen kernel.
\begin{lemma}[Smoothing effects]
    \label{lemma:smoothing-heat-Oseen}
    Let $1 \le p \le q \le \infty$ and $\beta \ge 0$. Let $f \in L^p(\bT^d)$. Then for any $0<t \le 1$,
    \begin{align*}
        \|(-\Delta)^\beta e^{-t(-\Delta)^\alpha} f\|_{L^q} &\lesssim t^{ -\frac{\beta}{\alpha}-\frac{d}{2\alpha}(\frac{1}{p}-\frac{1}{q}) } \|f\|_{L^p}, \\
        \|e^{-t(-\Delta)^\alpha} \bP \Div f\|_{L^q} &\lesssim t^{-\frac{1}{2\alpha}-\frac{d}{2\alpha}(\frac{1}{p}-\frac{1}{q})} \|f\|_{L^p}.
    \end{align*}
\end{lemma}
The $L^p$-estimates are classical, the reader can refer to \cite{Miao-Yuan-Zhang2008} for the proof. By applying the same Fourier multiplier bounds to each
Littlewood--Paley block, one obtains the corresponding estimates in $B^s_{q,r}(\bT^d)$ for
every $s\in \bR$ and $1\le q,r\le\infty$. 

\begin{proof}[Proof of Proposition \ref{prop:lwp-NSEa}]
In what follows, assume $0 \leq t \le T \le 1$.

First, consider \ref{item:lwp-NSEa-Lp}. The norm of the solution class $X_T$ is given by
\[ \|u\|_{X_T} := \sup_{0 \le t \le T} \|u(t)\|_{ L^p } + \sup_{0<t<T} t^{\frac{d}{2\alpha}(\frac{1}{p}-\frac{1}{\rho})} \|u(t)\|_{ L^\rho }. \]

The estimates for the linear part follow from Lemma \ref{lemma:smoothing-heat-Oseen} as
\[ \|e^{-t(-\Delta)^\alpha} u_{\initial}\|_{X_T} \lesssim \|u_{\initial}\|_{L^p}. \]
We show that the bilinear operator $B_\alpha$ defined by \eqref{e:Ba} satisfies
\begin{equation}
    \label{e:Ba-est-Lp}
    \|B_\alpha(u,v)\|_{X_T} \lesssim T^{1-\frac{1}{2\alpha}-\frac{d}{2\alpha p}} \|u\|_{X_T} \|v\|_{X_T}.
\end{equation}
Indeed, as $\rho \ge p'$, one has $(u \otimes v)(t) \in L^\varrho$ for $\frac{1}{\varrho}=\frac{1}{p}+\frac{1}{\rho} \le 1$. Then we apply Lemma \ref{lemma:smoothing-heat-Oseen} to get
\begin{align*}
    \|B_\alpha(u,v)(t)\|_{L^p} 
    &\lesssim \int_0^t (t-s)^{-\frac{1}{2\alpha}-\frac{d}{2\alpha}(\frac{1}{\varrho}-\frac{1}{p})} \|u \otimes v(s)\|_{L^\varrho} ds \\
    &\lesssim \int_0^t (t-s)^{-\frac{1}{2\alpha}-\frac{d}{2\alpha\rho}} s^{ -\frac{d}{2\alpha}(\frac{1}{p}-\frac{1}{\rho}) } \|u(s)\|_{L^p} \, s^{ \frac{d}{2\alpha}(\frac{1}{p}-\frac{1}{\rho}) } \|v(s)\|_{L^\rho} ds \\
    &\lesssim t^{1-\frac{1}{2\alpha}-\frac{d}{2\alpha p}} \|u\|_{X_T} \|v\|_{X_T}.
\end{align*}
The estimate for the other term in \eqref{e:Ba-est-Lp} follows similarly. As $p>\frac{d}{2\alpha-1}$, one has $1-\frac{1}{2\alpha}-\frac{d}{2\alpha p}>0$, so the desired properties directly follow from the classical Banach fixed-point argument. This proves \ref{item:lwp-NSEa-Lp}.

Next, we prove \ref{item:lwp-NSEa-B}. The norm of $Y_T$ is given by
\[ \|u\|_{Y_T} := \sup_{0 \le t \le T} \|u(t)\|_{ B^{-\theta}_{q,r} } + \sup_{0<t<T} t^{ \frac{\beta}{\alpha} } \|u(t)\|_{ B^{-\theta+2\beta}_{q,r} }. \]

The estimates for the linear part also follow from Lemma \ref{lemma:smoothing-heat-Oseen} as
\[ \| e^{-t(-\Delta)^\alpha} u_{\initial} \|_{B^{-\theta}_{q,r}} + t^{\frac{\beta}{\alpha}} \| (-\Delta)^{\beta} e^{-t(-\Delta)^\alpha} u_{\initial} \|_{B^{-\theta}_{q,r}} \lesssim \|u_{\initial}\|_{B^{-\theta}_{q,r}}. \]
We show that the bilinear operator $B_\alpha$ satisfies the estimate
\begin{equation}
    \label{e:Ba-est-B}
    \|B_\alpha(u,v)\|_{Y_T} \lesssim T^{1-\frac{\beta}{\alpha}-\frac{1}{2\alpha}} \|u\|_{Y_T} \|v\|_{Y_T}.
\end{equation}
The desired properties also follow from the fact that $1-\frac{\beta}{\alpha}-\frac{1}{2\alpha}>0$.

To prove \eqref{e:Ba-est-B}, we first apply Lemma \ref{lemma:smoothing-heat-Oseen} to get
\[ \|B_\alpha(u,v)(t)\|_{B^{-\theta}_{q,r}} \lesssim \int_0^t (t-s)^{-\frac{1}{2\alpha}} \|u \otimes v(s)\|_{B^{-\theta}_{q,r}} ds. \]
As $-\theta+2\beta-\frac{d}{q}>\theta$, we infer from Sobolev's embedding and the paraproduct estimates (see e.g., \cite[Theorem 4.37]{Sawano2018-book-Besov}) that
\[ \|u \otimes v(s)\|_{ B^{-\theta}_{q,r} } \lesssim \|u(s)\|_{ B^{-\theta}_{q,r} } \|v(s)\|_{ B^{-\theta+2\beta}_{q,r} } \lesssim s^{ -\frac{\beta}{\alpha} } \|u\|_{Y_T} \|v\|_{Y_T}. \]
As $\alpha>1/2$ and $1-\frac{\beta}{\alpha}-\frac{1}{2\alpha}>0$, gathering these estimates gives that
\[ \sup_{0 \le t \le T} \|B_\alpha(u,v)(t)\|_{ B^{-\theta}_{q,r} } \lesssim T^{1-\frac{\beta}{\alpha}-\frac{1}{2\alpha}} \|u\|_{Y_T} \|v\|_{Y_T}. \]
The estimate for the other term in \eqref{e:Ba-est-B} follows similarly. This proves \ref{item:lwp-NSEa-B}.

Finally, we show that $u \in C^\infty((0,T) \times \bT^d)$. By Sobolev's embedding, we may only consider the case $u_{\initial} \in B^{-\theta}_{q,r}$. For any sufficiently small $\epsilon>0$ and $t \ge 2\epsilon$, we infer from the identity
\[ u(t) = e^{-(t-\epsilon)(-\Delta)^\alpha} u(\epsilon) - \int_\epsilon^t e^{-(t-s)(-\Delta)^\alpha} \bP \Div (u \otimes u)(s) ds \]
and Lemma \ref{lemma:smoothing-heat-Oseen} that
\[ \|u(t)\|_{B^{-\theta+4\beta}_{q,r}} \lesssim \epsilon^{-\frac{2\beta}{\alpha}} \|u\|_{ L^\infty((0,T);B^{-\theta}_{q,r}) } + \int_\epsilon^t (t-s)^{-\frac{\beta}{\alpha}-\frac{1}{2\alpha}} \|u \otimes u(s)\|_{ B^{-\theta+2\beta}_{q,r} } ds. \]
Since $-\theta+2\beta>\frac{d}{q}$, $B^{-\theta+2\beta}_{q,r}$ is a multiplication algebra (see e.g., \cite[Theorem 2.8.3]{Triebel1983}), so one has
\[ \|u \otimes u(s)\|_{B^{-\theta+2\beta}_{q,r}} \lesssim \|u(s)\|_{B^{-\theta+2\beta}_{q,r}}^2 \lesssim s^{-\frac{2\beta}{\alpha}} \|u\|_{Y_T}^2. \]
As $\beta<(2\alpha-1)/2$ and $t \le 1$, gathering these estimates gives that
\[ \|u(t)\|_{B^{-\theta+4\beta}_{q,r}} \lesssim \epsilon^{-\frac{2\beta}{\alpha}} \|u\|_{Y_T} + \epsilon^{-\frac{2\beta}{\alpha}} \|u\|_{Y_T}^2. \]
So $u \in L^\infty((2\epsilon,T);B^{-\theta+4\beta}_{q,r})$. By iteration and arbitrariness of $\epsilon$, we know that $u(t)$ is smooth for any $0<t<T$, and hence $u \in C^\infty((0,T) \times \bT^d)$ by iteratively taking time derivatives on both sides of the equation $\partial_t u = -\Div(u \otimes u) - (-\Delta)^\alpha u - \nabla p$. This completes the proof.
\end{proof}

\begin{cor}[Unconditional uniqueness]
    \label{cor:uu-NSEa}
    Let $d \ge 2$, $\alpha>1/2$, and
    \[ \begin{cases}
        \frac{d}{2\alpha-1} < p \le \infty & \text{ if } \frac{1}{2} < \alpha \le \frac{d+2}{4} \\
        2 \le p \le \infty & \text{ if } \alpha>\frac{d+2}{4}
    \end{cases}. \]
    Then for every divergence-free $u_{\initial}\in L^p(\bT^d)$ and every $T>0$, there exists at most one mild solution to \eqref{e:NSEa} in $C([0,T];L^p(\bT^d))$.
\end{cor}

\begin{proof}
    The condition on $p$ ensures that $p \ge p'$ and $p>\frac{d}{2\alpha-1}$. When $p<\infty$, the statement hence follows from Proposition \ref{prop:lwp-NSEa} \ref{item:lwp-NSEa-Lp} by taking $\rho=p$.
     When $p=\infty$, the conclusion holds by embedding of $L^\infty$ into $L^{\tilde p}$ for $\tilde p<\infty$.
\end{proof}
%%%%% Sing-mild
\section{Stationary singular solutions and mild solutions}
\label{sec:sing->mild}

In this section, we prove Proposition \ref{prop:sing->mild}. As an application, we show uniqueness of solutions to \eqref{e:NS} in subcritical Lebesgue spaces.
\begin{proof}[Proof of Proposition \ref{prop:sing->mild}]
    Let $u \in \scrD'(\bT^d)$ be a singular solution to \eqref{e:NS}, for which $u \otimes u$ is well-defined as a paraproduct in $\DotH^{-s}$ for some $s>0$. It is clear that $U \in C([0,\infty);\scrD') \cap C^\infty((0,\infty);\scrD')$ with $U(0)=u$. As $U(t,x) = \I_{(0,\infty)}(t) \otimes u(x)$, we infer that for any divergence-free $\phi \in \Cc((0,\infty) \times \bT^d)$,
    \[ -\langle U,\partial_t \phi \rangle - \langle U \otimes U,\nabla \phi \rangle + \langle U,(-\Delta)^\alpha \phi \rangle = - \langle u \otimes u,\nabla \ovphi \rangle + \langle u,(-\Delta)^\alpha \ovphi \rangle = 0, \]
    where $\ovphi(x):=\int_0^\infty \phi(t,x) dt \in C^\infty(\bT^d)$. Thus, we have
    \[ \partial_t U + \bP \Div (U \otimes U) + (-\Delta)^\alpha U = 0 \quad \text{ in } \scrD'((0,\infty) \times \bT^d). \]
    Moreover, as $\bP \Div(U \otimes U) \in C([0,\infty);H^{-s-1})$, we know that
    \[ v(t) := \int_0^t e^{-(t-s)(-\Delta)^\alpha} \bP \Div(U \otimes U)(s) ds, \quad t>0 \]
    satisfies the equation
    \[ \partial_t v + (-\Delta)^\alpha v = \bP \Div (U \otimes U) \quad \text{ in } \scrD'((0,\infty) \times \bT^d). \]
    As $U \otimes U \in C([0,\infty);\DotH^{-s})$, we get from the smoothing effects of the heat equation that $v \in C([0,\infty);H^{-s-1+2\alpha}) \cap C^1((0,\infty);H^{-s-1})$ with $v(0)=0$.

    Write $w=U+v$. It satisfies $\partial_t w + (-\Delta)^\alpha w = 0$ in $\scrD'((0,\infty) \times \bT^d)$, and $w \in C([0,\infty);\scrD') \cap C^1((0,\infty);\scrD')$ with $w(0)=u$. Hence, we infer from uniqueness of heat solutions in the class $C([0,\infty);\scrD') \cap C^1((0,\infty);\scrD')$ that $w(t) = e^{-t(-\Delta)^\alpha} u$, and thus,
    \[ U(t) = w(t) - v(t) = e^{-t(-\Delta)^\alpha} u - \int_0^t e^{-(t-s)(-\Delta)^\alpha} \bP \Div(U \otimes U)(s) ds \]
    in $\scrD'((0,\infty) \times \bT^d)$ as desired. This completes the proof.
\end{proof}

\begin{cor}[Uniqueness of stationary solutions]
    \label{cor:sing-sol-Lp=0}
    Let $d \ge 2$, $\alpha>1/2$, and
    \[ \begin{cases}
        \frac{d}{2\alpha-1} < p \le \infty & \text{ if } \frac{1}{2} < \alpha \le \frac{d+2}{4} \\
        2 \le p \le \infty & \text{ if } \alpha>\frac{d+2}{4}
    \end{cases}. \]
    If $u \in L^p(\bT^d)$ is a singular solution to \eqref{e:NS}, then $u$ is zero.
\end{cor}

\begin{proof}
    Define $U(t):=u$ for $t \ge 0$. Proposition \ref{prop:sing->mild} yields that $U \in C([0,\infty);L^p)$ is a mild solution to \eqref{e:NSEa} with $U(0)=u$. Moreover, we infer from unconditional uniqueness (cf. Corollary \ref{cor:uu-NSEa}) and Proposition \ref{prop:lwp-NSEa} that $U \in C^\infty((0,\infty) \times \bT^d)$. So $u \in C^\infty(\bT^d)$. Therefore, testing the stationary equation \eqref{e:NS} against $u$ itself gives
    \[ \|(-\Delta)^{\alpha/2} u\|_{L^2}^2 = \int (-\Delta)^\alpha u \cdot u = \int (u \otimes u) : \nabla u = 0. \]
    Since $u$ is of mean zero, we conclude that $u=0$.
\end{proof}
%%%%% Iteration
\section{Iterative procedure}
\label{sec:iteration}

The proofs of the main theorems are based on the following iterative procedure. Consider the stationary fractional Navier-Stokes-Reynolds system:
\begin{equation}
    \label{e:NSR}
    \begin{cases}
        \Div(u \otimes u) + (-\Delta)^\alpha u + \nabla p = \Div R \\
        \Div u = 0 \\
        \hatu(0) = 0
    \end{cases}
\end{equation}
where $R$ is defined on $\bT^d$ with values in the space of symmetric matrices, called the \emph{Reynolds stress}. Here, we do not require that $R$ is trace-free. Indeed, if $(u,R,p)$ is a solution to \eqref{e:NSR}, then by taking 
\[ \ringR := R-\frac{1}{d} \tr(R) \bI, \quad \ringp := p - \frac{1}{d} \tr(R), \]
one gets another solution $(u,\ringR,\ringp)$ to \eqref{e:NSR} so that $\ringR$ is trace-free. Moreover, the normalized pressure $p$ is uniquely determined by the equation
\[ \Delta p = \Div \Div (R - u \otimes u), \quad \fint p = 0. \]
For simplicity, we abbreviate a solution triple $(u,R,p)$ to \eqref{e:NSR} as $(u,R)$. 

\begin{prop}[Main iteration]
    \label{prop:iteration}
    Let $d\ge 2$ and $\alpha>0$.
    There exists $s>0$ depending only on $d$ and $\alpha$ such that the following holds.
    
    Fix any $1 < p < 2$, $0<\theta<1$, $\delta>0$, and a dyadic integer $N$. Given a smooth solution $(u,R)$ to \eqref{e:NSR} with $\supp \hatu \subset B(0,N)$, there exists another smooth solution $(\ovu,\ovR)$ such that
    \begin{equation}
        \label{e:ovR-est}
        \|\ovR\|_{H^{-s}} < \delta,
    \end{equation}
    and moreover, the velocity perturbation $w:=\ovu-u$ satisfies
    \begin{enumerate}[label=\normalfont(\roman*)]
        \item \label{item:prop-hatw-supp}
        The Fourier support of $w$ is away from that of $u$, in the sense that there exists a dyadic number $D>50N$ so that
        \[ \supp \hatw \subset \left\{ m \in \bZ^d: \frac{1}{2} D N < |m| < \frac{9}{10} D N \right\}; \]

        \item \label{item:prop-w-Lp}
        The $L^p$-estimate and the $B^{-\theta}_{\infty,1}$-estimate
        \[ \|w\|_{L^p} < \delta, \quad \|w\|_{B^{-\theta}_{\infty,1}} < \delta; \]

        \item \label{item:prop-w-paraproduct}
        The paraproduct estimate
        \[ \|w \otimes w\|_{\DotH^{-s}} + \sum_{M \le 2N} \left(\|P_M u \otimes w\|_{\DotH^{-s}} + \|w \otimes P_M u\|_{\DotH^{-s}} \right)< \delta + \|R\|_{H^{-s}}. \]
    \end{enumerate}
\end{prop}

The proof of Proposition \ref{prop:iteration} is deferred to Sections \ref{sec:convex-integration}, \ref{sec:est}, and \ref{sec:pf-prop}. The next proposition is concerned with the iteration in the $L^2$-case.

\begin{prop}[$L^2$-iteration]
    \label{prop:L2-iteration}
    Let $0<\alpha<\frac{d+1}{4}$ and $N$ be a dyadic integer. Then there exists a constant $A \ge 1$, depending only on $d$ and $\alpha$, such that the following holds.
    
    For any $\delta,\delta'>0$, given a smooth solution $(u,R)$ to \eqref{e:NSR}, there exists another smooth solution $(\ovu,\ovR)$ such that 
    \begin{equation}
        \label{e:ovR-est-L1}
        \|\ovR\|_{L^1} < \delta^2,
    \end{equation}
    and moreover, the velocity perturbation $w:=\ovu-u$ satisfies 
    \begin{enumerate}[label=\normalfont(\roman*)]
        \item The $L^2$-estimate
        \begin{equation}
            \label{e:w-est-L2}
            \|w\|_{L^2} < A\|R\|_{L^1}^{1/2}+\delta;
        \end{equation}
        
        \item Almost Fourier support separation, in the sense that
        \[ \|P_{\le N} w\|_{L^2} < \delta'. \]
    \end{enumerate}
\end{prop}

The proof of Proposition \ref{prop:L2-iteration} is presented in Section \ref{sec:pf-L2-iteration}.
%%%%% Pf-thms
\section{Proofs of Main Theorems}
\label{sec:pf-thms}

Admitting Propositions \ref{prop:iteration} and \ref{prop:L2-iteration}, let us first prove Theorem \ref{thm:main}.

\subsection{Proof of Theorem \ref{thm:main} admitting Proposition \ref{prop:iteration}}
\label{ssec:pf-main}
Let $1 \le p < 2$, $\theta>0$, and $q,r \in [1,\infty]$. We prove the statements for $\bP B^{-\theta}_{q,r}(\bT^d)$, and the case in $\bP L^p(\bT^d)$ follows similarly. Moreover, in what follows, we only consider the case $r \ne \infty$. For $r=\infty$, the statement follows from weak*-density of $B^{-\theta}_{q,1}$ in $B^{-\theta}_{q,\infty}$.

Pick an arbitrary vector field $v \in \bP B^{-\theta}_{q,r}$. Fix $\epsilon>0$. We shall construct a stationary singular solution 
\[ u \in \bigcap_{1 \le p < 2} L^p \cap \bigcap_{\theta>0,~q,r \in [1,\infty]} B^{-\theta}_{q,r}, \]
such that
\begin{equation}
    \label{e:density-u-v}
    \|u-v\|_{B^{-\theta}_{q,r}} < \epsilon.
\end{equation}
Pick $N_0>1$ and a non-zero, mean-zero, divergence-free smooth vector field $u_0 \in C^\infty$ such that $\supp \hatu_0 \subset B(0,N_0)$ and
\begin{equation}
    \label{e:density-v-u0}
    \|v-u_0\|_{B^{-\theta}_{q,r}} < \epsilon/2.
\end{equation}
For example, if $v$ is non-zero, then one may pick $u_0 = P_{\le N_0} v$ for large enough $N_0$. If $v=0$, then one may pick $u_0$ with a single non-zero Fourier mode (e.g., $\supp \hatu_0 \subset \bS^{d-1} \cap \bZ^d$) and $0<\|u_0\|_{L^q}<\epsilon/2$.

Denote by $\cR$ the anti-divergence operator (introduced in \cite{DeLellis-Szekelyhidi2013-convex-int}) from smooth vector fields $C^\infty(\bT^d;\bR^d)$ to the space of symmetric trace-free $d \times d$ matrices, given by the formula
\begin{equation}
    \label{e:R-antidiv}
    (\cR f)_{i,j} := \cR_{i,j,k} f_k, \quad f \in C^\infty(\bT^d;\bR^d),
\end{equation}
where
\[ \cR_{i,j,k} := \frac{2-d}{d-1} \Delta^{-2} \partial_i \partial_j \partial_k - \frac{1}{d-1} \Delta^{-1} \partial_k \delta_{i,j}  + \Delta^{-1} \partial_i \delta_{j,k} + \Delta^{-1} \partial_j \delta_{i,k}. \]
Here, $\delta_{i,j}$ is the Kronecker symbol, and the inverse Laplacian $\Delta^{-1}$ of a smooth function $f \in C^\infty(\bT^d)$ is given by the formula
\[ \Delta^{-1} f(x) := \sum_{m \in \bZ^d \setminus \{0\}} -\frac{1}{4\pi^2|m|^2} \hatf(m) e^{2\pi i m \cdot x}. \]
Note that $w=\Delta^{-1} f$ is the unique solution $w \in C^\infty(\bT^d)$ to the equation
\begin{equation}
    \label{e:Delta^-1}
    \Delta w = f - \fint f, \quad \fint w = 0.
\end{equation}
Direct computation shows that 
\begin{equation}
    \label{e:divR}
    \Div \cR f = f - \fint f.
\end{equation}

Define
\[ R_0 := \cR \left( \Div(u_0 \otimes u_0) + (-\Delta)^\alpha u_0 \right), \]
and set
\[ w_0:=u_0, \quad D_0:=1. \]
Then $(u_0,R_0)$ is a smooth solution to \eqref{e:NSR} with $\supp \hatu_0 \subset B(0,N_0)$.

We use iteration to construct a sequence of smooth solutions $(u_n,R_n)$ for $n \ge 1$ as follows. Let $s>0$ be the constant given by Proposition \ref{prop:iteration}. Set $N_{n-1} = N_0 D_0 \cdots D_{n-1}$. Given $(u_{n-1},R_{n-1})$ with $\supp \hatu_{n-1} \subset B(0,N_{n-1})$, we apply Proposition \ref{prop:iteration} to $(u_{n-1},R_{n-1})$ with the parameters
\begin{equation}
    \label{e:delta_n}
    \delta_n := 2^{-n-2} \epsilon, \quad p_n := 2-2^{-n}(2-p), \quad \theta_n := \min\{2^{-n} \theta, 1/2 \}.
\end{equation}
to get a smooth solution $(u_n,R_n)$ with
\[ \|R_n\|_{H^{-s}}<\delta_n. \]
Moreover, denote by $w_n:=u_n-u_{n-1}$ the velocity perturbation. Then there exists a dyadic number $D_n>50N_{n-1}$ such that
\[ \supp \hatw_n \subset \left\{m \in \bZ^d:\frac{1}{2} D_n N_{n-1} < |m| <\frac{9}{10} D_n N_{n-1} \right\}, \]
and $w_n$ satisfies the estimates
\begin{equation}
    \label{e:density-wn-est}
    \|w_n\|_{L^{p_n}} < \delta_n, \quad \|w_n\|_{B^{-\theta_n}_{\infty,1}} < \delta_n,
\end{equation}
and
\[ \|w_n \otimes w_n\|_{\DotH^{-s}} + \sum_{M \le 2N_{n-1}} \big(\|P_M u_{n-1} \otimes w_n\|_{\DotH^{-s}} + \|w_n \otimes P_M u_{n-1}\|_{\DotH^{-s}} \big)< \delta_n + \|R_{n-1}\|_{H^{-s}}. \]

For any $1 \le \tilp < 2$, $\tiltheta>0$, and $\tilq,\tilr \in [1,\infty]$. Then there exists $n_0 \ge 1$ such that for $n \ge n_0$, one has $p_n \ge \tilp$ and $\theta_n \le \tiltheta$, so
\[ \|w_n\|_{L^\tilp} + \|w_n\|_{B^{-\tiltheta}_{\tilq,\tilr}} \le \|w_n\|_{L^{p_n}} + \|w_n\|_{B^{-\theta_n}_{\infty,1}} < 2\delta_n. \] 
Thus, $(u_n)$ is a Cauchy sequence in $L^\tilp$ and $B^{-\tiltheta}_{\tilq,\tilr}$. Denote by $u$ the limit of $(u_n)$ in $\scrD'$. Then
\[ u \in \bigcap_{1 \le p < 2} L^p \cap \bigcap_{\theta>0,~q,r \in [1,\infty]} B^{-\theta}_{q,r}, \]
and
\begin{equation}
    \label{e:un->u}
    u_n \rightarrow u ~ \text{ in } L^p \text { and } B^{-\theta}_{q,r} \quad \text{for all } 1 \le p <2, ~ \theta>0, ~ q,r \in [1,\infty].
\end{equation}

We claim that $u$ is a singular solution to \eqref{e:NS}. Indeed, we first verify that $u \otimes u$ is well-defined as a paraproduct in $\DotH^{-s}$. By Fourier support separation in Proposition \ref{prop:iteration}, we observe that
\[ P_N u_{n-1} = \begin{cases}
    P_N u_0 & \text{ if } N \le 2N_0 \\
    w_j & \text{ if } N=N_j \text{ for } 1 \le j \le n-1 \\
    0 & \text{ otherwise }
\end{cases}, \]
and hence,
\begin{equation}
    \label{e:PNu}
    P_N u = 
    \begin{cases}
        P_M u_0 & \text{ if } M \le 2N_0 \\
        w_n & \text{ if } N=N_n, \text{ for } n \ge 1 \\
        0 & \text{ otherwise }
    \end{cases}.
\end{equation} 
Then using \eqref{e:ovR-est}, we get
\begin{align*}
    \sum_{N,M} & \|P_N u \otimes P_M u\|_{\DotH^{-s}} \\
    &= \sum_{N,M \le 2N_0} \|P_M u_0 \otimes P_N u_0\|_{\DotH^{-s}} + \sum_{n \ge 1} \|w_n \otimes w_n\|_{\DotH^{-s}} \\
    &\quad + \sum_{n \ge 1} \sum_{M \le 2N_{n-1}} \left( \|P_M u_{n-1} \otimes w_n\|_{\DotH^{-s}} + \|w_n \otimes P_M u_{n-1}\|_{\DotH^{-s}} \right) \\
    &\le \sum_{N,M \le 2N_0} \|P_M u_0 \otimes P_N u_0\|_{\DotH^{-s}} + \sum_{n \ge 1} \left( \delta_n + \|R_{n-1}\|_{H^{-s}} \right) \\
    &\le \sum_{N,M \le 2N_0} \|P_M u_0 \otimes P_N u_0\|_{\DotH^{-s}} + \|R_0\|_{H^{-s}} + 2\sum_{n \ge 1} \delta_n < \infty.
\end{align*}
This proves \eqref{e:u*u-condition}, so $u \otimes u$ is well-defined as a paraproduct in $\DotH^{-s}$.

Next, consider \eqref{e:sing-sol-id}. Since $(u_n,R_n)$ is a smooth solution to \eqref{e:NSR}, we have that for any divergence-free $\varphi \in C^\infty$,
\begin{equation}
    \label{e:un-Rn-weak}
    -\left\langle u_n \otimes u_n, \nabla \varphi \right\rangle + \langle u_n,(-\Delta)^\alpha \varphi \rangle = -\langle R_n,\nabla \varphi \rangle.
\end{equation} 
We get from \eqref{e:u*u-limit} that
\[ \lim_{n \to \infty} u_n \otimes u_n = \lim_{n \to \infty} P_{\le N_n} u \otimes P_{\le N_n} u = u \otimes u \quad \text{ in } \DotH^{-s}. \]
Moreover, as $\|R_n\|_{H^{-s}} < \delta_n$, one gets
\[ R_n \to 0 \quad \text{ in } H^{-s}. \]
Gathering these limits with \eqref{e:un->u}, we obtain \eqref{e:sing-sol-id} by taking limits $n \to \infty$ in \eqref{e:un-Rn-weak}. 

Finally, we get from \eqref{e:PNu} and \eqref{e:un->u} that $\hatu(0)=0$ and $u$ is weakly divergence free. The claim hence follows.

Let us finish by verifying \eqref{e:density-u-v}. Indeed, using \eqref{e:density-v-u0} and \eqref{e:density-wn-est}, we get
\begin{align*}
    \|v-u\|_{B^{-\theta}_{q,r}} 
    &\le \|v-u_0\|_{B^{-\theta}_{q,r}} + \|u-u_0\|_{B^{-\theta}_{q,r}} \\
    &\le \frac{\epsilon}{2} + \sum_{n \ge 1} \|w_n\|_{B^{-\theta_n}_{\infty,1}} \le \frac{\epsilon}{2} + \sum_{n \ge 1} \delta_n < \epsilon
\end{align*}
as desired. This completes the proof.

\subsection{Proof of Theorem \ref{thm:weak} admitting Proposition \ref{prop:L2-iteration}}
\label{sec:pf-thm-weak}
The proof resembles that in Section \ref{ssec:pf-main}. Let $0<\alpha<(d+1)/4$. Let $0<\epsilon<1$. Pick a non-zero divergence-free smooth vector field $u_0 \in C^\infty(\bT^d)$ of mean zero, with 
\begin{equation}
    \label{e:L2-u0}
    \supp \hatu_0 \in \bZ^d \cap \bS^{d-1}, \quad \|u_0\|_{C^{\lceil 2\alpha \rceil}} < \epsilon_0,
\end{equation}
where $\epsilon_0 \in (0,1)$ is a constant to be determined. Define the Reynolds stress
\[ R_0:=\cR(\Div(u_0 \otimes u_0) + (-\Delta)^\alpha u_0). \]
It is clear that
\[ \|R_0\|_{L^1} \le C \left( \|u_0\|_{C^{\lceil 2\alpha \rceil}} + \|u_0\|_{C^{\lceil 2\alpha \rceil}}^2 \right) < 2C \epsilon_0, \]
where the constant $C \ge 1$ depends only on $d$ and $\alpha$. Next, we use Proposition \ref{prop:L2-iteration} for iteration. Namely, for $n \ge 1$, given $(u_{n-1},R_{n-1})$ a smooth solution to \eqref{e:NSR}, we apply Proposition \ref{prop:L2-iteration} with parameter
\[ \delta_n := A^{-1} 2^{-n-4} \epsilon, \quad N=10, \quad \delta_n' := 2^{-n-2} \|u_0\|_{L^2} \]
to get a new smooth solution $(u_n,R_n)$ to \eqref{e:NSR} with
\[ \|R_n\|_{L^1} < \delta_n^2, \quad \|w_n\|_{L^2} < A \|R_{n-1}\|_{L^1}^{1/2} + \delta_n, \quad \|P_{\le N} w_n\|_{L^2} < 2^{-n-2} \|u_0\|_{L^2} \]
for $n \ge 1$, where $w_n:=u_n-u_{n-1}$ denotes the velocity perturbation. For $n=1$, one has
\begin{equation}
    \label{e:L2-w1}
    \|w_1\|_{L^2} < A\|R_0\|_{L^1}^{1/2} + \delta_1 < 2C A \epsilon_0^{1/2} + 2^{-5} \epsilon.
\end{equation}
For $n \ge 2$, using $\|R_{n-1}\|_{L^1} < \delta_{n-1}^2$, one has
\begin{equation}
    \label{e:L2-wn}
    \|w_n\|_{L^2} < A \delta_{n-1} + \delta_n < 2^{-n-2} \epsilon,
\end{equation}
so $(u_n)$ is a Cauchy sequence in $L^2$. Denote by $u$ the limit of $(u_n)$. Then we have
\[ u_n \to u \text{ in } L^2, \quad u_n \otimes u_n \to u \otimes u \text{ in } L^1, \quad R_n \to 0 \text{ in } L^1. \]
We thus conclude that $u \in L^2$ is a singular, hence weak solution to \eqref{e:NS}. Moreover, we also have
\[ \|P_{\le N} u\|_{L^2} \ge \|u_0\|_{L^2}-\sum_{n \ge 1} \|P_{\le N} w_n\|_{L^2} \ge 1/2\|u_0\|_{L^2} > 0, \]
so $u$ is non-trivial. Combining \eqref{e:L2-u0}, \eqref{e:L2-w1}, and \eqref{e:L2-wn} gives that
\[ \|u\|_{L^2} \le \|u_0\|_{L^2} + \sum_{n \ge 1} \|w_n\|_{L^2} < 3CA\epsilon_0^{1/2} + \epsilon/4. \]
Taking $\epsilon_0$ sufficiently small with $3CA\epsilon_0^{1/2} < \epsilon/2$ proves \eqref{e:weak-L2-bd}. 

Let us finish by proving non-uniqueness of weak solutions to \eqref{e:NSEa} in $L^\infty([0,\infty);L^2) \cap C_{\ssfw}([0,\infty);L^2)$. For any $\epsilon>0$, pick the initial data $u_{\initial}$ as a stationary weak solution in $L^2(\bT^d)$ with $\|u_{\initial}\|_{L^2}<\epsilon$. Then Proposition \ref{prop:sing->mild} implies that $u(t) := u_{\initial}$ for $t \ge 0$ is a weak solution to \eqref{e:NSEa}. 

Meanwhile, we take $v$ as the classical Leray-Hopf weak solution evolving from $u_{\initial} \in L^2$, which belongs to $L^\infty((0,\infty);L^2) \cap C_{\ssfw}([0,\infty);L^2)$ and obeys the energy inequality (see \cite{Leray1934,Hopf1951} for $\alpha=1$ and straightforward extension to all $\alpha>0$). Note that since $v$ is non-trivial, the energy obeys $\|v(t)\|_{L^2} < \|v_{\initial}\|_{L^2}$ for every positive $t$, but the energy of $u$ is constant as $u$ is stationary. Therefore, $u(t)$ and $v(t)$ are distinct for any $t>0$. This completes the proof.
    
\subsection{Proofs of Theorems \ref{thm:nonunique-mild-NSE} and \ref{thm:nonunique-mild-NSEa}}
\label{ssec:pf-nonunique}
We prove Theorem \ref{thm:nonunique-mild-NSEa}. Theorem \ref{thm:nonunique-mild-NSE} corresponds to the case $\alpha=1$. Let us start with \ref{item:nonunique-NSEa-B}. Note that the Besov space $B^{-\theta}_{q,1}(\bT^d)$ embeds into $B^{-\theta_1}_{q_1,r_1}(\bT^d)$ when $-\theta_1 \le -\theta$, $-\theta_1-\frac{d}{q_1} \le -\theta-\frac{d}{q}$, and $1 \le r_1 \le \infty$. Therefore, it suffices to prove the statements for
\[ \max\left\{\frac{d}{2\alpha-1},2\right\} < q < \infty, \quad 0<\theta<\frac{1}{2} \left( 2\alpha-1-\frac{d}{q} \right), \quad r=1. \]

We first construct the initial data $u_{\initial} \in B^{-\theta}_{q,1}$. Let $f$ be an arbitrary vector field in $\bP B^{-\theta}_{q,1}$. For any $\epsilon>0$, Theorem \ref{thm:main} yields that there exists a stationary singular solution $u_{\initial} \in \bP B^{-\theta}_{q,1}(\bT^d)$ such that $\|u_{\initial} - f\|_{B^{-\theta}_{q,1}} < \epsilon$.

Then we construct two different mild solutions $u$ and $v$ to \eqref{e:NSEa} with the same initial data $u_{\initial}$. Indeed, Proposition \ref{prop:sing->mild} yields that 
\[ u(t) := u_{\initial}, \quad t \ge 0 \]
is a mild solution to \eqref{e:NSEa}, and particularly, $u \in C([0,\infty);B^{-\theta}_{q,1})$. Meanwhile, Proposition \ref{prop:lwp-NSEa} says that there exist $T>0$ and a mild solution
\[ v \in C([0,T];B^{-\theta}_{q,1}) \cap C^\infty( (0,T) \times \bT^d) \]
to \eqref{e:NSEa} with $v(0)=u_{\initial}$.

Observe that $u_{\initial} \notin C^\infty$. Otherwise, Corollary \ref{cor:sing-sol-Lp=0} implies that $u_{\initial} = 0$, a contradiction. So $u(t) = u_{\initial} \ne v(t)$ for any $t \in (0,T]$. This proves \ref{item:nonunique-NSEa-B}.

Statement \ref{item:nonunique-NSEa-Lp} follows from a verbatim adaptation of the above argument with a slightly larger exponent. 
Fix any $1 \le p < 2$. Since $\alpha>(d+2)/4$, one has $\frac{d}{2\alpha-1}<2$, so we can choose
\[
\tilde p \in \left(\max\left\{p,\frac{d}{2\alpha-1}\right\},\,2\right).
\]

For any $f \in \bP L^p(\bT^d)$ and $\epsilon>0$, pick a smooth divergence-free and mean-zero approximation $g \in C^\infty$ first with $\|g-f\|_{L^p} < \epsilon/2$. Then by density (cf. Theorem \ref{thm:main}), there exists a stationary singular solution $u_{\initial} \in L^{\tilp}$ to \eqref{e:NS} with $\|u_{\initial} - g\|_{L^\tilp} < \epsilon/2$. Thus, one has $\|u_{\initial} - f\|_{L^p} \le \|u_{\initial} - g\|_{L^\tilp} + \|g-f\|_{L^p} < \epsilon$.

Take $u(t)=u_{\initial}$ for $t \ge 0$.
Meanwhile, Proposition \ref{prop:lwp-NSEa} yields that there exist $T>0$ and a mild solution
\[
v \in C([0,T];L^{\tilde p}) \cap C^\infty((0,T)\times \bT^d)
\]
to \eqref{e:NSEa} with $v(0)=u_{\initial}$. In particular, $v \in C([0,T];L^p)$. Again, since $u_{\initial} \notin C^\infty$, we conclude that $u(t)=u_{\initial}\ne v(t)$ for any $t \in (0,T]$.
This completes the proof.

\subsection{Proof of Theorem \ref{thm:unique-NS-L2B-1}}
\label{ssec:pf-unique-NS-L2B-1}
For $\alpha > \frac{d+2}{4}$, this is Corollary \ref{cor:sing-sol-Lp=0}. In what follows, we only consider $\alpha \le \frac{d+2}{4}$. The proof is based on estimates of the energy flux from high to low modes, introduced by \cite{Cheskidov-Constantin-Friedlander-Shvydkoy2008-Onsager}. For any dyadic integer $N\ge 2$, we decompose $u$ as 
\[
u = u_{\le N}+u_{>N}, \quad \text{ where } ~ u_{\le N} := P_{\le N} u ~ \text{ and } ~ u_{>N} := P_{>N} u.
\]
Testing the stationary equation against $P_{\le N}^2u$ (equivalently, applying $P_{\le N}$ to \eqref{e:NSE} and then testing against $u_{\le N}$), and using that $P_{\le N}$ commutes with $\nabla$ and preserves divergence-free vector fields, we obtain
\begin{equation}\label{eq:Pi-leN-def}
\Pi_{\le N}
:=\int |(-\Delta)^{\alpha/2} u_{\le N}|^2
= \int P_{\le N}(u\otimes u):\nabla u_{\le N}.
\end{equation}

We now exploit the skew-symmetry of the transport term. Introduce the commutator
\begin{align*}
    r_{\le N}(u,u)(x)
    &:=P_{\le N}(u\otimes u)(x)-u_{\le N} \otimes u_{\le N}(x) \\
    &= \int \check\chi_N(y)\,\bigl(u(x-y)-u_{\le N}(x)\bigr)\otimes\bigl(u(x-y)-u_{\le N}(x)\bigr)\,dy,
\end{align*}
where $\check\chi_N$ denotes the kernel of $P_{\le N}$ (i.e., the inverse Fourier transform of $m\mapsto \chi(m/N)$). 
Since $\Div u_{\le N}=0$, we have the exact cancellation
\[
\int (u_{\le N}\otimes u_{\le N}):\nabla u_{\le N} = 0,
\]
and hence, from \eqref{eq:Pi-leN-def},
\begin{equation}\label{eq:Pi-leN}
\Pi_{\le N}=\int r_{\le N}(u,u):\nabla u_{\le N}. 
\end{equation}

We shall show that the energy concentrated in high modes
\[ E_N := \sum_{M>N} \|P_M u\|_{L^2}^2 \]
vanishes for sufficiently large $N$. Therefore, $u_{>N}=0$ and $u$ is smooth. Testing \eqref{e:NS} against $u$ and using $\Div u=0$, we get
\[ \|(-\Delta)^{\alpha/2}u\|_{L^2}^2 = \int (u \otimes u):\nabla u = 0. \]
Hence $u$ is constant, and the mean-zero condition implies $u=0$.

To this end, the proof is divided into two cases: $0<\alpha \le 1$ and $1 < \alpha \le \frac{d+2}{4}$. \\

\paragraph{Case 1: $0<\alpha \le 1$}
Bernstein's inequality yields that
\[
\|\nabla u_{\le N}\|_{L^\infty}
\lesssim
\sum_{M \le N} M \|P_M u\|_{L^\infty}
= N^{2\alpha} \sum_{M \le N} \left( \frac{M}{N} \right)^{2\alpha} M^{1-2\alpha}\|P_M u\|_{L^\infty}
=: N^{2\alpha} b_N.
\]
Next, we show that
\begin{equation}\label{eq:rleN-L1}
\|r_{\le N}(u,u)\|_{L^1} \lesssim N^{-2\alpha} \Pi_{\le N} + E_N.
\end{equation}
Indeed, by Minkowski's inequality, one has
\begin{align*}
    \|r_{\le N}(u,u)\|_{L^1}
    &\le \int |\check\chi_N(y)| \|u(\cdot-y)-u_{\le N}(\cdot)\|_{L^2}^2 dy \\
    &\lesssim \int |\check\chi_N(y)| \left( \|u_{\le N}(\cdot-y)-u_{\le N}(\cdot)\|_{L^2}^2 + \|u_{>N}\|_{L^2}^2 \right) dy
\end{align*}
Using orthogonality and the mean value estimate, we infer that
\begin{align*}
    \|u_{\le N}(\cdot-y)-u_{\le N}\|_{L^2}^2
    &\lesssim \sum_{M \le N} \|P_M u(\cdot-y) - P_M u\|_{L^2} \\
    &\lesssim |y|^2\sum_{M\le N} M^2 \|P_M u\|_{L^2}^2 \lesssim |y|^2 N^{2-2\alpha} \Pi_{\le N}.
\end{align*}
Thus, \eqref{eq:rleN-L1} follows from the kernel bounds $\int |\check\chi_N(y)|\,|y|^2\,dy\lesssim N^{-2}$ and $\int |\check\chi_N(y)|\,dy\lesssim 1$.
Applying H\"older's inequality to \eqref{eq:Pi-leN} and gathering the above estimates gives that
\begin{equation}\label{eq:Pi-leN-ineq}
\Pi_{\le N} \le C b_N \Pi_{\le N} + C N^{2\alpha} b_N E_N,
\end{equation}
where the constant $C>0$ depends only on $d$.

Let us exploit the smallness of $b_N$. For a reference dyadic number $N_0\ge2$, Fubini's theorem yields that
\begin{align*}
    B_{N_0} := \sum_{N \ge N_0} b_N 
    &= \sum_M \sum_{N \ge \max\{M,N_0\}} \left( \frac{M}{N} \right)^{2\alpha} M^{1-2\alpha} \|P_M u\|_{L^\infty} \\
    &\lesssim \sum_M \min\left\{ 1, \left( \frac{M}{N_0} \right)^{2\alpha} \right\} M^{1-2\alpha} \|P_M u\|_{L^\infty}.
\end{align*}
Since $\sum_M M^{1-2\alpha} \|P_M u\|_{L^\infty} = \|u\|_{B^{1-2\alpha}_{\infty,1}} <\infty$, Lebesgue's dominated convergence theorem yields that 
\begin{equation}
    \label{e:limBN0=0}
    \lim_{N_0 \to \infty} B_{N_0} = 0.
\end{equation}
Pick $N_0$ large enough so that $CB_{N_0}\le 1/2$. In particular, for any $N \ge N_0$, as $Cb_N \le CB_{N_0} \le 1/2$, we hence get from \eqref{eq:Pi-leN-ineq} that
\begin{equation}\label{eq:Pi-leN-tail}
\Pi_{\le N} \lesssim N^{2\alpha} b_N E_N.
\end{equation}

Then we show that the energy concentrated on high modes is null. Indeed, using \eqref{eq:Pi-leN-tail}, we get that for $N_1 \ge N_0$,
\begin{align*}
    E_{N_1} 
    &= \sum_{N>N_1} \|P_N u\|_{L^2}^2 \lesssim \sum_{N>N_1} N^{-2\alpha} \Pi_{\le N} \lesssim \sum_{N>N_1} b_N E_N \lesssim B_{N_1} E_{N_1},
\end{align*}
where the implicit constant also depends only on $d$. Thanks to \eqref{e:limBN0=0}, for large enough $N_1$, we conclude that
\[ E_{N_1} = 0. \]
This concludes the case $\alpha \le 1$.
\\

\paragraph{Case 2: $1 < \alpha \le \frac{d+2}{4}$}
The proof is similar to that of Case 1. Denote by $p$ the H\"older conjugate of $q=\frac{d}{2\alpha-2}$. H\"older's inequality yields that
\[ \Pi_{\le N} \le \|\nabla u_{\le N}\|_{L^q} \|r_{\le N}(u,u)\|_{L^p}. \]

Using Bernstein's inequality again, one gets that
\[ \|\nabla u_{\le N}\|_{L^q} \lesssim N^{2} \sum_{M \le N} \left( \frac{M}{N} \right)^2 M^{-1} \|P_M u\|_{L^q} =: N^2 b_N. \]
Also, write $B_{N_0} := \sum_{N \ge N_0} b_N$. Using Lebesgue's dominated convergence theorem again, we get from the condition $\|u\|_{B^{-1}_{q,1}} < \infty$ that
\begin{equation}
    \label{e:unique-alpha>1-limBN0=0}
    \lim_{N_0 \to \infty} B_{N_0} = 0.
\end{equation}
Meanwhile, we claim that
\begin{equation}\label{eq:rleN-Lp-case2}
    \|r_{\le N}(u,u)\|_{L^p} \lesssim N^{-2}\Pi_{\le N} + N^{2\alpha-2}E_N.
\end{equation}
Suppose it holds. Then gathering these estimates gives that
\[ \Pi_{\le N} \lesssim b_N \Pi_{\le N} + N^{2\alpha} b_N E_N, \]
which agrees with \eqref{eq:Pi-leN-ineq}. Thanks to \eqref{e:unique-alpha>1-limBN0=0}, the same reasoning as in Case 1 yields that
\[ E_{N_1} = 0 \]
for sufficiently large $N_1$ as desired. In summary, we only need to show \eqref{eq:rleN-Lp-case2} in what follows.

To this end, we introduce a finer decomposition of the commutator $r_{\le N}(u,u)$. Define
\[ \ell_y(x) := u_{\le N}(x)-u_{\le N}(x+y). \]
The kernel representation of $r_N(u,u)$ thus reads as
\begin{align*}
    r_N(u,u)(x) 
    &= \int \check\chi_N(y) (\ell_y \otimes \ell_y + \ell_y \otimes u_{>N} + u_{>N} \otimes \ell_y + u_{>N} \otimes u_{>N})(x-y) dy \\
    &=: r_N^{l,l}(u,u)(x) + r_N^{l,h}(u,u)(x) + r_N^{h,l}(u,u)(x) + r_N^{h,h}(u,u)(x).
\end{align*}

We estimate these terms respectively. First, consider the low-low contribution $r_N^{l,l}(u,u)$. H\"older's inequality yields that
\[ \|r_N^{l,l}(u,u)\|_{L^p} \le \int |\check\chi_N(y)| \|\ell_y\|_{L^{2p}}^2 dy. \]
Using the Littlewood-Paley theory and the mean value theorem, we get
\begin{equation}
    \label{e:ly-L2p-est}
    \begin{aligned}
        \|\ell_y\|_{L^{2p}}^2 
        &\lesssim |y|^2 \sum_{M \le N} M^2\|P_M u\|_{L^{2p}}^2 \\ 
        &\lesssim |y|^2 \sum_{M\le N}M^{2\alpha}\|P_Mu\|_{L^2}^2 \lesssim |y|^2 \Pi_{\le N}.
    \end{aligned}
\end{equation}
The last inequality is due to Bernstein's inequality and $d(\frac12-\frac1{2p})=\alpha-1$. Thus, we obtain
\[ \|r_N^{l,l}(u,u)\|_{L^p} \lesssim N^{-2}\Pi_{\le N}.
\]

Next, we consider the high-high contribution $r^{h,h}_N(u,u)$. Note that
\[ r_N^{h,h}(u,u) = \check\chi_N \ast (u_{>N} \otimes u_{>N}) = P_{\le N} (u_{>N} \otimes u_{>N}), \]
so using Young's inequality and Bernstein's inequality, we infer that
\[ \|r_N^{h,h}(u,u)\|_{L^p} \lesssim \|P_{\le N}(u_{>N}\otimes u_{>N})\|_{L^p} \lesssim N^{d(1-\frac1p)}\|u_{>N}\|_{L^2}^2 \lesssim N^{2\alpha-2} E_N. \]

It remains to estimate the mixed terms. We take $r_N^{l,h}(u,u)$ for example and the other follows similarly. Note that
\[ r_N^{l,h}(u,u) = P_{\le N} (u_{\le N} \otimes u_{>N}) - u_{\le N} \otimes P_{\le N} u_{>N}. \]
In particular, we infer that the Fourier transform of $r_N^{l,h}(u,u)$ is supported in $\{m \in \bZ^d:|m| \le 2N\}$. Take $r=\frac{2p}{p+1}$ so that $\frac{1}{r}=\frac{1}{2p}+\frac{1}{2}$. We get from Bernstein's inequality and \eqref{e:ly-L2p-est} that
\begin{align*}
    \|r_N^{l,h}(u,u)\|_{L^p} 
    &\lesssim N^{d(\frac1r-\frac1p)} \|r_N^{l,h}(u,u)\|_{L^r} \\
    &\lesssim N^{\alpha-1} \int |\check\chi_N(y)| \|\ell_y\|_{L^{2p}} \|u_{>N}\|_{L^2} dy \\
    &\lesssim N^{\alpha-2} \Pi_{\le N}^{1/2} E_N^{1/2} \lesssim N^{-2}\Pi_{\le N} + N^{2\alpha-2}E_N.
\end{align*}
In the second inequality, we use the fact that $d(\frac1r-\frac1p) = \alpha-1$. Gathering these estimates gives \eqref{eq:rleN-Lp-case2}. This completes the proof.
%%%%% Convex Integration
\section{Convex integration}
\label{sec:convex-integration}

In this section, we construct the velocity perturbation $w$ in Proposition \ref{prop:iteration} via convex integration. We use Mikado flows introduced by Daneri and Székelyhidi in \cite[Lemma 2.3]{Daneri-Szekelyhidi2017-Mikado} as building blocks of our convex integration scheme, see also \cite{Buckmaster-Vicol2019-ems-survey,Cheskidov-Luo2022-LpLinfty,Ashkarian-Bhargava-Gismondi-Novack2025}. In what follows, we fix
\[ \boxed{\theta \in (0,1), \quad e:=2\left \lceil \theta^{-1} \right \rceil > 2 }.  \]

\subsection{Mikado flows}
\label{ssec:mikado}
Let us start with a geometric lemma by De Lellis and Székelyhidi \cite[Lemma 3.2]{DeLellis-Szekelyhidi2013-convex-int}, which originated from Nash \cite{Nash1954-C1embedding}. 

Denote by $B_{1/2}(\bI)$ the ball of positive definite symmetric $d \times d$ matrices (with respect to the matrix norm), centered at the identity matrix $\bI$, of radius $1/2$. There exist a finite set $\Lambda \subset \bZ^d$ and smooth functions $\Gamma_k:B_{1/2}(\bI) \to \bR$ for $k \in \Lambda$ such that
\[ R = \sum_{k \in \Lambda} \Gamma_k^2(R) k \otimes k, \quad \forall R \in B_{1/2}(\bI). \] 

For $k \in \Lambda$, pick $p_k \in (0,1)^d$ and denote by $\ell_k$ the line passing through $p_k$ in the direction $k$ in $\bT^d$, i.e., $\ell_k = \{p_k + tk \in \bT^d: t \in \bR\}$. One may adjust $p_k$ so that $p_k \notin \ell_{-k}$ when both $k$ and $-k$ lie in $\Lambda$.
 
Let $\mu \ge 1$ be a concentration parameter. Let $\psi:[1/2,1] \to \bR$ be a non-zero smooth function such that $\psi_k:\bT^d \to \bR$, defined by
\[ \psi_k(x) := c_k \mu^{\frac{1}{2}(d-1)} \psi(\mu \dist(x,\ell_k)), \]
with $c_k>0$ a normalization constant, satisfies that
\[ \fint \psi_k = 0, \qquad \text{and} \qquad \fint \psi_k^2 = 1. \]

Define the \emph{Mikado flows} $W_k:\bT^d \to \bR^d$ for $k \in \Lambda$ by
\[ W_k(x) := \psi_k(x) k. \]
Notice that $W_k$ is of mean zero, divergence-free, and satisfies
\[ \fint W_k \otimes W_k = k \otimes k. \]
Also, $W_k$ is a solution to the pressureless stationary Euler equation, namely, $\Div(W_k \otimes W_k) = 0$. Moreover, define the anti-symmetric tensor $\Omega_k:\bT^d \to \mat_d(\bR)$ by
\[ \Omega_k := k \otimes \nabla \Delta^{-1} \psi_k - \nabla \Delta^{-1} \psi_k \otimes k. \]
Using $(k \cdot \nabla) \psi_k = 0$ and $\fint \psi_k = 0$, we get
\[ \Div \Omega_k = \Div(\nabla \Delta^{-1} \psi_k) k - (k \cdot \nabla) \nabla \Delta^{-1} \psi_k = \psi_k k - 0 = W_k. \]

We collect estimates of the Mikado flows in the following lemma. The reader can refer to \cite[Theorem 4.3]{Cheskidov-Luo2022-LpLinfty} for the proof.
\begin{lemma}[Estimates on the Mikado flows]
    \label{lemma:W_k}
    Let $d \ge 2$ and $1 \le p \le \infty$. Then for any $m \ge 0$,
    \begin{align*}
        \|\nabla^m W_k\|_{L^p} &\lesssim_m \mu^m \mu^{(d-1)(\frac{1}{2}-\frac{1}{p})}, \\
        \|\nabla^m \Omega_k\|_{L^p} &\lesssim_m \mu^m \mu^{-1+(d-1)(\frac{1}{2}-\frac{1}{p})}.
    \end{align*}
    Moreover, for $l \in \Lambda$ but $l \ne k$, it holds that
    \[ \|W_k \otimes W_l\|_{L^p} \lesssim \mu^{d-1-\frac{d}{p}}. \]
\end{lemma}

\subsection{Velocity perturbation}
\label{ssec:w}
The main point of our construction is to localize the Fourier transform of the velocity perturbation, while also keeping intermittency. To this end, we define the principal part of the perturbation $w^{(p)}$ by
\[ w^{(p)}(x) := \sum_{k \in \Lambda} a_k(x) W_k(\gamma x) \cos(2\pi\sigma_k k^\perp \cdot x), \]
where $\gamma,\sigma_k \ge 1$ are oscillation parameters, and $k^\perp \in \bZ^d \setminus \{0\}$ is a lattice point that is perpendicular to $k$ with $|k^\perp| \le |k|$. The amplitude function $a_k:\bT^d \to \bR$ is given by
\[ a_k(x) := (8d^2 \|R\|_{L^\infty})^{1/2} \, \Gamma_k \left( \bI-\frac{R(x)}{4d^2\|R\|_{L^\infty}} \right), \]
in case $R(x)$ is not identically zero. If $R \equiv 0$, set $a_k \equiv 0$ for all $k \in \Lambda$. Note that $\bI-\frac{R(x)}{4d^2\|R\|_{L^\infty}}$ lies in $B_{1/2}(\bI)$. The role of $\sigma_k$ is to ensure that all $\sigma_k k^\perp$ lie in the same dyadic shell that is sufficiently far away from the origin, particularly when $k^\perp$ are not of the same length, as we shall see in Lemma \ref{lemma:arithmetic-sigma}. 

\subsubsection{Localization corrector}
We define the frequency localization corrector $w^{(l)}$ with respect to a dyadic integer $\lambda \ge 1$ by
\begin{align*}
    w^{(l)}(x) := &-\sum_{k \in \Lambda} a_k(x) P_{>\lambda} (W_k(\gamma \cdot))(x) \cos(2\pi\sigma_k k^\perp \cdot x) \\
    &-\sum_{k \in \Lambda} P_{>\lambda} a_k(x) P_{\le \lambda} (W_k(\gamma \cdot))(x) \cos(2\pi\sigma_k k^\perp \cdot x).
\end{align*}
Thus,
\begin{equation}
    \label{e:w(p)+w(l)}
    w^{(p)} + w^{(l)} = \sum_{k \in \Lambda} P_{\le \lambda} a_k P_{\le \lambda} (W_k(\gamma \cdot)) \cos(2\pi \sigma_k k^\perp \cdot x).
\end{equation}

Given a dyadic integer $\lambda \ge 1$, the following arithmetic lemma determines $\sigma_k$ so that all $\sigma_k k^\perp$ lie in the same dyadic shell.

\begin{lemma}[Arithmetic lemma]
    \label{lemma:arithmetic-sigma}
    For any dyadic integer $\lambda \ge 1$, there exist integers $\{\sigma_k\}_{k \in \Lambda}$ and a dyadic number $\sigma>50\lambda^{e}$ so that
    \[ \bigcup_{k \in \Lambda} \big(B(\sigma_k k^\perp, 2\lambda) \cup B(-\sigma_k k^\perp, 2\lambda) \big) \subset \left\{ m \in \bZ^d:\frac{1}{2} \sigma < |m| < \frac{9}{10} \sigma \right\}. \]
\end{lemma}

\begin{proof}
    Write $c_\Lambda := \prod_{k \in \Lambda} |k^\perp|$. We claim that it suffices to verify that there exist an integer $b \ge 45\lambda^{e}$ and a dyadic number $\sigma$ such that
    \begin{equation}
        \label{e:11/9-9/5-bc-sigma}
        \frac{11}{9} bc_{\Lambda} < \sigma < \frac{9}{5} bc_{\Lambda}.
    \end{equation}
    Indeed, suppose \eqref{e:11/9-9/5-bc-sigma} holds. Then pick $\sigma_k := \left \lceil \nicefrac{bc_{\Lambda}}{|k^\perp|} \right \rceil$. As $b \ge 45\lambda^{e} > 40\lambda$, we have $2\lambda < bc_{\Lambda}/20$, and moreover,
    $\left| |\sigma_k k^\perp|-bc_\Lambda \right| \le |k^\perp| \le c_\Lambda < bc_{\Lambda}/20$. Hence, we have
    \[ \frac{1}{2} \sigma < \frac{9}{10} bc_{\Lambda} < |\sigma_k k^\perp|-2\lambda < |\sigma_k k^\perp|+2\lambda < \frac{11}{10} bc_{\Lambda} < \frac{9}{10} \sigma. \]
    As $b \ge 45\lambda^{e}$, we have $\sigma > 55 c_\Lambda \lambda^{e}> 50\lambda^{e}$ as desired.

    Let us finish by proving \eqref{e:11/9-9/5-bc-sigma}. Consider the translation map $T$ on $\bT \simeq \bR/\bZ$ defined by
    \[ T(x) := x+\log_2(3\lambda^{e}) \mod 1. \]
    Since $\lambda$ is dyadic, $\log_2(3\lambda^{e})=\log_2 3 + m$ for some integer $m$, hence it is irrational. The (positive) orbit of any $x \in \bT$, $\{T^n x:n \ge 0\}$, is dense in $\bT$. In particular, there exists $l \ge 1$ such that $T^l(\log_2(55c_\Lambda)) \in \left( 1-\log_2\left(81/55 \right),1 \right)$. Equivalently, there exists an integer $j \ge 1$ with
    \[ \log_2(55 c_\Lambda) + l\log_2(3\lambda^e) < j < \log_2(81 c_\Lambda) + l\log_2(3\lambda^e). \]
    Then \eqref{e:11/9-9/5-bc-sigma} follows by taking $\sigma=2^j$ and $b=45 (3\lambda^e)^l$.
\end{proof}

\subsubsection{Divergence-free corrector}
Since $w^{(p)}+w^{(l)}$ is not divergence-free, we need to introduce the divergence-free corrector $w^{(d)}$. Write
\[ \Psi_k := P_{\le \lambda} (\psi_k(\gamma \cdot)) \cos(2\pi \sigma_k k^\perp \cdot x), \]
and
\[ \bfW_k := \Psi_k k = P_{\le \lambda} ( W_k(\gamma \cdot) ) \cos(2\pi \sigma_k k^\perp \cdot x). \]
Define the anti-symmetric tensor $\bfOmega_k$ by
\[ \bfOmega_k := k \otimes \nabla \Delta^{-1} \Psi_k - \nabla \Delta^{-1} \Psi_k \otimes k, \]
and the divergence-free corrector $w^{(d)}$ by
\[ w^{(d)} := \sum_{k \in \Lambda} \nabla P_{\le \lambda} a_k : \bfOmega_k. \]

\begin{lemma}
    \label{lemma:w}
    Given $\lambda,\mu,\gamma \ge 1$ and $\sigma_k \ge 1$ as in Lemma \ref{lemma:arithmetic-sigma}, the velocity perturbation
    \begin{equation}
        \label{e:w}
        w := w^{(p)} + w^{(l)} + w^{(d)}
    \end{equation}
    is divergence-free, and there exists a dyadic integer $\sigma>50\lambda^e$ so that
    \begin{equation}
        \label{e:supp-hatw}
        \supp \hatw \subset \left\{ m \in \bZ^d:\frac{1}{2} \sigma < |m| < \frac{9}{10} \sigma \right\}.
    \end{equation}
\end{lemma}

\begin{proof}
    We first verify that $w$ is divergence-free. Using $(k \cdot \nabla) \psi_k = 0$ and $k \cdot k^\perp = 0$, we get $(k \cdot \nabla) \Psi_k = 0$ and hence,
    \[ \Div \bfOmega_k = \Div(\nabla \Delta^{-1} \Psi_k) k - (k \cdot \nabla) \nabla \Delta^{-1} \Psi_k = \Psi_k k - 0 = \bfW_k. \]
    Here, we use the fact that $\Psi_k$ is of mean zero, as $\supp \widehat{\Psi}_k \subset B(\sigma_k k^\perp,\lambda) \cup B(-\sigma_k k^\perp,\lambda)$ is away from 0. Thus, we get
    \[ w = \sum_{k \in \Lambda} P_{\le \lambda} a_k \bfW_k + \nabla P_{\le \lambda} a_k : \bfOmega_k = \Div \sum_{k \in \Lambda} P_{\le \lambda} a_k \bfOmega_k. \]
    Since $P_{\le \lambda} a_k \bfOmega_k$ is anti-symmetric, we obtain $\Div w=0$.

    Next, to prove \eqref{e:supp-hatw}, we notice that
    \[
    \supp \hatw \subset \bigcup_{k \in \Lambda}\Big( B(\sigma_k k^\perp,2\lambda)\cup B(-\sigma_k k^\perp,2\lambda)\Big) =: \frakS.
    \]
    Indeed, we get from \eqref{e:w(p)+w(l)} that the Fourier transform of $w^{(p)} + w^{(l)}$ is supported in $\frakS$. Moreover, $\widehat{\Psi}_k$ is supported in $B(\sigma_k k^\perp, \lambda) \cup B(-\sigma_k k^\perp,\lambda)$, and so is $\widehat{\bfOmega}_k$. Thus, the Fourier transform of $w^{(d)}$ is supported in $\frakS$, and so is that of $w$. Then \eqref{e:supp-hatw} follows from Lemma \ref{lemma:arithmetic-sigma} as
    \[ \supp \hatw \subset \frakS \subset \left\{ m \in \bZ^d:\frac{1}{2} \sigma < |m| < \frac{9}{10} \sigma \right\}. \]
    This completes the proof.
\end{proof}

\subsection{Reynolds stress}
To make $(\ovu,\ovR)$ a solution to \eqref{e:NSR}, the corrector $w=\ovu-u$ and the new Reynolds stress $\ovR$ should satisfy the equation
\begin{equation}
    \label{e:ovR-eq}
    \Div \ovR = \Div (R+w \otimes w) + (-\Delta)^\alpha w + \Div(u \otimes w+w \otimes u) + \nabla P
\end{equation}
for certain pressure $P \in C^\infty$.

Given $w$ defined by \eqref{e:w}, let us construct $\ovR$. First, consider the non-linear term. Denote by $w^{(c)}$ the full corrector
\[ \wcor := \wloc + \wdiv. \]
So $w=\wprin+\wcor$ and
\[ R+w \otimes w = R+\wprin \otimes \wprin + \wprin \otimes \wcor + \wcor \otimes w. \]
Define the correction error $\Rcor$ by
\[ \Rcor := \wprin \otimes \wcor + \wcor \otimes w. \]
Then we decompose the tensor $\wprin \otimes \wprin$ to the diagonal part and the off-diagonal part as
\begin{align*}
    &R + \wprin \otimes \wprin \\
    &\quad = R + \frac{1}{2}\sum_{k \in \Lambda} a_k^2 (1+\cos(4\pi\sigma_k k^\perp \cdot x)) W_k(\gamma x) \otimes W_k(\gamma x) \\
    &\qquad + \sum_{\substack{k,l \in \Lambda\\k \ne l}} a_k a_l \cos(2\pi\sigma_k k^\perp \cdot x) \cos(2\pi\sigma_l l^\perp \cdot x) W_k(\gamma x) \otimes W_l(\gamma x)
\end{align*}
Define the off-diagonal error $\Roff$ by
\[ \Roff := \sum_{\substack{k,l \in \Lambda\\k \ne l}} a_k a_l \cos(2\pi\sigma_k k^\perp \cdot x) \cos(2\pi\sigma_l l^\perp \cdot x) W_k(\gamma x) \otimes W_l(\gamma x), \]
and the distortion error $\Rdis$ by
\[ \Rdis := \frac{1}{2}\sum_{k \in \Lambda} a_k^2 \cos(4\pi\sigma_k k^\perp \cdot x) W_k(\gamma x) \otimes W_k(\gamma x). \]
We further decompose the tensor $W_k(\gamma x) \otimes W_k(\gamma x)$ to the average part and the oscillation part. Since $\fint W_k \otimes W_k = k \otimes k$, we obtain
\begin{align*}
    R + \wprin \otimes \wprin
    &= R + \frac{1}{2}\sum_{k \in \Lambda} a_k^2 k \otimes k \\
    &\quad + \frac{1}{2}\sum_{k \in \Lambda} a_k^2 \left(W_k(\gamma x) \otimes W_k(\gamma x)-\fint W_k \otimes W_k \right) + \Rdis + \Roff.
\end{align*}
Define the oscillation error $\Rosc$ by
\[ \Rosc = \frac{1}{2}\sum_{k \in \Lambda} a_k^2 \left(W_k(\gamma x) \otimes W_k(\gamma x)-\fint W_k \otimes W_k \right). \]
Moreover, one has
\[ \frac{1}{2}\sum_{k \in \Lambda} a_k^2 k \otimes k = 4d^2 \|R\|_{L^\infty} \sum_{k \in \Lambda} \Gamma_k^2 \left( \bI-\frac{R}{4d^2\|R\|_{L^\infty}} \right) k \otimes k = 4d^2 \|R\|_{L^\infty} \bI - R. \]
Therefore, we obtain
\begin{equation}
    \label{e:R+w*w}
    R + w \otimes w = 4d^2 \|R\|_{L^\infty} \bI + \Rosc + \Rdis + \Roff + \Rcor.
\end{equation}

Finally, we define the linear error $\Rlin$ by
\[ \Rlin := \cR( (-\Delta)^\alpha w + \Div(u \otimes w+w \otimes u)), \]
and the new Reynolds stress $\ovR$ by
\begin{equation}
    \label{e:ovR}
    \ovR := \Rosc + \Rdis + \Roff + \Rcor + \Rlin.
\end{equation}

\begin{lemma}
    \label{lemma:ovu,ovR-sol}
    Let $(u,R)$ be a smooth solution to \eqref{e:NSR} with the pressure $p$. Then for any $\lambda,\mu,\gamma \ge 1$, $(\ovu,\ovR)$ is another smooth solution to \eqref{e:NSR} with the same pressure $p$, where the velocity perturbation $w:=\ovu-u$ and $\ovR$ are defined by \eqref{e:w} and \eqref{e:ovR}.
\end{lemma}

\begin{proof}
    First, we prove that $\ovR$ is symmetric. Indeed, \eqref{e:R+w*w} implies that $\ovR-\Rlin = R + w \otimes w - 4d^2 \|R\|_{L^\infty} \bI$ is symmetric. By definition of the anti-divergence operator $\cR$, $\Rlin$ is also symmetric, and so is $\ovR$.
    
    Next, we show that $\ovR$ satisfies \eqref{e:ovR-eq} with $P=0$. By \eqref{e:R+w*w}, we get
    \begin{align*}
        \Div \ovR 
        &= \Div(R+w \otimes w - 4d^2\|R\|_{L^\infty} \bI) + \Div \Rlin \\
        &= \Div(R+w \otimes w) + \Div \cR( (-\Delta)^\alpha w + \Div(u \otimes w+w \otimes u)) \\
        &= \Div(R+w \otimes w) + (-\Delta)^\alpha w + \Div(u \otimes w+w \otimes u).
    \end{align*}
    The last equality holds thanks to \eqref{e:divR} and the fact that $(-\Delta)^\alpha w + \Div(u \otimes w+w \otimes u)$ is of mean zero.
\end{proof}
%%%%% Estimates
\section{Estimates on the perturbation and the stress}
\label{sec:est}

This section is devoted to estimates on $w$ and $\ovR$ based on the following four parameters.
\begin{center}
    \begin{tabular}{r@{\hspace{3em}}c@{\hspace{3em}}l}
        $\lambda \ge 2$ & dyadic integer & reference frequency\\
        $\mu \ge 1$ & -- & concentration parameter\\
        $\gamma \ge 1$ & dyadic integer & slow oscillation\\
        $\sigma > 50\lambda^e$ & dyadic integer & fast oscillation\\
    \end{tabular}
\end{center}
In what follows, $C_u$ denotes a constant that depends only on $(u,R)$.

Let us start with a lemma on boundedness of the Littlewood-Paley projections on $L^p$.
\begin{lemma}
    \label{lemma:PN}
    Let $1 \le p \le \infty$. For any $f \in C^\infty(\bT^d)$ and any dyadic $N$, it holds that
    \begin{align*}
        \|P_{\le N} f\|_{L^p} &\lesssim \|f\|_{L^p}, \\
        \|P_{>N} f\|_{L^p} &\lesssim N^{-1} \|\nabla f\|_{L^p}.
    \end{align*}
\end{lemma}

\begin{proof}
    The first inequality is a direct consequence of Bernstein's inequality (see e.g., \cite[Lemma 2.1]{Bahouri-Chemin-Danchin2011-book}). The second also follows from Bernstein's inequality as
    \begin{align*}
        \|P_{>N} f\|_{L^p}
        &\le \sum_{M \ge 2N} \|P_M f\|_{L^p} \lesssim \sum_{M \ge 2N} M^{-1} \|\nabla P_M f\|_{L^p} \\
        &\lesssim \sum_{M \ge 2N} M^{-1} \|\nabla f\|_{L^p} \lesssim N^{-1} \|\nabla f\|_{L^p}. \qedhere
    \end{align*}
\end{proof}

\begin{prop}[Estimates on $w$]
    \label{prop:w-est}
    Let $1 \le p \le \infty$ and $\theta \in \bR$. There exists a constant $C_1 \ge 1$ depending only on $(u,R)$ with the following estimates
    \begin{align}
        \label{e:wp-est}
        \|\wprin\|_{L^p} &\le C_1 \mu^{(d-1)(\frac{1}{2}-\frac{1}{p})}, \\
        \label{e:wl-est}
        \|\wloc\|_{L^\infty} &\le C_1 \lambda^{-1} \gamma \mu^{1+\frac{1}{2}(d-1)}, \\
        \label{e:wd-est}
        \|\wdiv\|_{L^\infty} &\le C_1 \sigma^{-1} \mu^{\frac{1}{2}(d-1)}, \\
        \label{e:w-est-besov}
        \|w\|_{B^{-\theta}_{\infty,1}} &\le C_1 \sigma^{-\theta} \mu^{\frac{1}{2}(d-1)}.
    \end{align}
\end{prop}

\begin{proof}
    First, \eqref{e:wp-est} directly follows from H\"older's inequality as
    \begin{align*}
        \|\wprin\|_{L^p}
        &\le \sum_{k \in \Lambda} \|a_k\|_{L^\infty} \|W_k(\gamma \cdot)\|_{L^p} \|\cos(2\pi \sigma_k k^\perp \cdot x)\|_{L^\infty} \\
        &\le \sum_{k \in \Lambda} \|a_k\|_{L^\infty} \|W_k\|_{L^p} \le C_u \mu^{(d-1)(\frac{1}{2}-\frac{1}{p})}.
    \end{align*}

    Next, for \eqref{e:wl-est}, using Lemma \ref{lemma:PN}, we obtain
    \begin{align*}
        \|a_k P_{>\lambda}(W_k(\gamma \cdot)) \cos(2\pi \sigma_k k^\perp \cdot x)\|_{L^\infty}
        &\le \|a_k\|_{L^\infty} \|P_{>\lambda}(W_k(\gamma \cdot))\|_{L^\infty} \\
        &\lesssim \|a_k\|_{L^\infty} \lambda^{-1} \|\nabla (W_k(\gamma \cdot))\|_{L^\infty} \\
        &\le C_u \lambda^{-1} \gamma \mu^{1+\frac{1}{2}(d-1)}.
    \end{align*}
    Moreover, as $\gamma \ge 1$ and $\mu \ge 1$, we have
    \begin{align*}
        \|P_{>\lambda} a_k P_{\le \lambda}(W_k(\gamma \cdot)) \cos(2\pi \sigma_k k^\perp \cdot x)\|_{L^\infty}
        &\lesssim \|P_{>\lambda} a_k\|_{L^\infty} \|W_k\|_{L^\infty} \\
        &\lesssim \lambda^{-1} \|a_k\|_{C^1} \mu^{\frac{1}{2}(d-1)} \\
        &\le C_u \lambda^{-1} \gamma \mu^{1+\frac{1}{2}(d-1)}.
    \end{align*}

    For \eqref{e:wd-est}, we also use Bernstein's inequality to get
    \begin{align*}
        \|\bfOmega_k\|_{L^\infty} 
        &\lesssim \|\nabla \Delta^{-1} \Psi_k\|_{L^\infty} \lesssim \sigma^{-1} \|\Psi_k\|_{L^\infty} \\
        &\le \sigma^{-1} \|P_{\le \lambda} (\psi_k(\gamma \cdot)) \cos(2\pi \sigma_k k^\perp \cdot x)\|_{L^\infty} \lesssim \sigma^{-1} \mu^{\frac{1}{2}(d-1)}.
    \end{align*}
    Here, we also use the fact that $\widehat{\Psi}_k$ is supported in the annulus $\{m \in \bZ^d:\sigma/2 < |m| < 9\sigma/10\}$. Thus, we have
    \[ \|\wdiv\|_{L^\infty} \le \sum_{k \in \Lambda} \|\nabla P_{\le \lambda} a_k\|_{L^\infty} \|\bfOmega_k\|_{L^\infty} \le C_u \sigma^{-1} \mu^{\frac{1}{2}(d-1)}. \]

    Finally, to prove \eqref{e:w-est-besov}, we infer from \eqref{e:w(p)+w(l)} that
    \[ \|\wprin+\wloc\|_{L^\infty} \le \sum_{k \in \Lambda} \|P_{\le \lambda} a_k\|_{L^\infty} \|P_{\le \lambda} (W_k(\gamma \cdot))\|_{L^\infty} \le C_u \mu^{\frac{1}{2}(d-1)}. \]
    Since $\hatw$ is supported in the annulus $\{m \in \bZ^d:\sigma/2 < |m| < 9\sigma/10\}$ (cf. Lemma \ref{lemma:w}), we obtain
    \begin{align*}
        \|w\|_{B^{-\theta}_{\infty,1}} 
        &\lesssim \sigma^{-\theta} \|w\|_{L^\infty} \le \sigma^{-\theta} \left( \|\wprin+\wloc\|_{L^\infty} + \|\wdiv\|_{L^\infty} \right) \\
        &\le C_u \sigma^{-\theta} \mu^{\frac{1}{2}(d-1)}.
    \end{align*}
    This completes the proof.
\end{proof}

\begin{prop}[Estimates on $\ovR$]
    \label{prop:R-est}
    Let $s \ge \frac{d}{2}+2\alpha$. Then there exists a constant $C_2 \ge 1$ depending only on $(u,R)$ with the following estimates
    \begin{align}
        \label{e:Rosc-est}
        \|\Rosc\|_{H^{-s}} &\le C_2 \left( \gamma^{-s} \mu^{\frac{1}{2}(d-1)} + \gamma^{-1} \right), \\
        \label{e:Rdis-est}
        \|\Rdis\|_{H^{-s}} &\le C_2 \left( \sigma^{-s} \mu^{\frac{1}{2}(d-1)} + \lambda^{-1} \gamma \mu \right), \\
        \label{e:Roff-est}
        \|\Roff\|_{L^1} &\le C_2 \mu^{-1}, \\
        \label{e:Rcor-est}
        \|\Rcor\|_{L^1} &\le \|\wcor\|_{L^\infty} (\|w\|_{L^1} + \|\wprin\|_{L^1}), \\
        \label{e:Rlin-est}
        \|\Rlin\|_{H^{-s}} &\le C_2 \|w\|_{L^p}, \quad \text{ for all } 1<p<\infty.
    \end{align}
\end{prop}

\begin{proof}
    First consider \eqref{e:Rosc-est}. Write
    \begin{equation}
        \label{e:Tk}
        T_k(x) := W_k(\gamma x) \otimes W_k(\gamma x) - \fint W_k \otimes W_k,
    \end{equation}
    and we decompose $\Rosc$ as
    \[ \Rosc = \frac{1}{2} \sum_{k \in \Lambda} a_k^2 T_k = \frac{1}{2} \sum_{k \in \Lambda} \left(P_{\le \gamma/2} (a_k^2) T_k + P_{>\gamma/2} (a_k^2) T_k \right). \]
    Since $\hatT_k$ is supported in $\{m \in \bZ^d:|m| \ge \gamma\}$, one has
    \begin{align*}
        \|P_{\le \gamma/2} (a_k^2) T_k\|_{H^{-s}} 
        &\lesssim \gamma^{-s} \|P_{\le \gamma/2} (a_k^2) T_k\|_{L^2} \lesssim \gamma^{-s} \|P_{\le \gamma/2} (a_k^2)\|_{L^\infty} \|T_k\|_{L^2} \\
        &\lesssim \gamma^{-s} \|a_k\|_{L^\infty}^2 \|W_k\|_{L^4}^2 \le C_u \gamma^{-s} \mu^{\frac{1}{2}(d-1)}.
    \end{align*}
    Using Lemma \ref{lemma:PN}, we also have
    \begin{align*}
        \|P_{>\gamma/2} (a_k^2) T_k\|_{L^1} 
        &\le \|P_{>\gamma/2} (a_k^2)\|_{L^\infty} \|T_k\|_{L^1} \\
        &\lesssim \gamma^{-1} \|\nabla(a_k^2)\|_{L^\infty} \|W_k\|_{L^2}^2 \le C_u \gamma^{-1}.
    \end{align*}
    Since $-s-d/2<-d$, Sobolev's embedding theorem yields that $L^1$ embeds into $H^{-s}$, so gathering these estimates gives \eqref{e:Rosc-est}.

    Next, to prove \eqref{e:Rdis-est}, we decompose $\Rdis$ into three parts
    \begin{align*}
        \Rdis
        &:=\frac{1}{2}\sum_{k \in \Lambda} P_{\le \lambda}(a_k^2) P_{\le \lambda} (W_k(\gamma \cdot) \otimes W_k(\gamma \cdot)) \cos(4\pi \sigma_k k^\perp \cdot x) \\
        &\quad +\frac{1}{2}\sum_{k \in \Lambda} P_{>\lambda}(a_k^2) P_{\le \lambda} (W_k(\gamma \cdot) \otimes W_k(\gamma \cdot)) \cos(4\pi \sigma_k k^\perp \cdot x) \\
        &\quad +\frac{1}{2}\sum_{k \in \Lambda} a_k^2 P_{>\lambda} (W_k(\gamma \cdot) \otimes W_k(\gamma \cdot)) \cos(4\pi \sigma_k k^\perp \cdot x) \\
        &:= I_1 + I_2 + I_3.
    \end{align*}
    Note that $\hatI_1$ is supported in $\{m \in \bZ^d:\sigma<|m|<9\sigma/5\}$, so we have
    \begin{align*}
        \|I_1\|_{H^{-s}} 
        &\lesssim \sigma^{-s} \sum_{k \in \Lambda} \|P_{\le \lambda} (a_k^2)\|_{L^\infty} \|P_{\le \lambda} (W_k(\gamma \cdot) \otimes W_k(\gamma \cdot))\|_{L^2} \\
        &\lesssim \sigma^{-s} \sum_{k \in \Lambda} \|a_k\|_{L^\infty}^2 \|W_k\|_{L^4}^2 \le C_u \sigma^{-s} \mu^{\frac{1}{2}(d-1)}.
    \end{align*}
    The estimates of $I_2$ and $I_3$ also follow from Lemma \ref{lemma:PN} as
    \begin{align*}
        \|I_2\|_{L^1} 
        &\lesssim \sum_{k \in \Lambda} \|P_{>\lambda}(a_k^2)\|_{L^\infty} \|P_{\le \lambda} (W_k(\gamma \cdot) \otimes W_k(\gamma \cdot))\|_{L^1} \\
        &\lesssim \lambda^{-1} \sum_{k \in \Lambda} \|\nabla (a_k^2)\|_{L^\infty} \|W_k\|_{L^2}^2 \le C_u \lambda^{-1} \le C_u \lambda^{-1} \gamma \mu,
    \end{align*}
    where the last inequality holds as $\gamma \ge 1$ and $\mu \ge 1$, and
    \begin{align*}
        \|I_3\|_{L^1} 
        &\lesssim \sum_{k \in \Lambda} \|a_k^2\|_{L^\infty} \|P_{>\lambda} (W_k(\gamma \cdot) \otimes W_k(\gamma \cdot))\|_{L^1} \\
        &\lesssim \lambda^{-1} \sum_{k \in \Lambda} \|a_k\|_{L^\infty}^2 \|\nabla ((W_k \otimes W_k)(\gamma \cdot))\|_{L^1} \\
        &\lesssim \lambda^{-1} \gamma \sum_{k \in \Lambda} \|a_k\|_{L^\infty}^2 \|\nabla W_k\|_{L^2} \|W_k\|_{L^2} \le C_u \lambda^{-1} \gamma \mu.
    \end{align*}
    Using again the fact that $L^1$ embeds into $H^{-s}$, we combine these estimates to obtain the estimate desired in \eqref{e:Rdis-est}.

    For \eqref{e:Roff-est}, it follows from H\"older's inequality and Lemma \ref{lemma:W_k} as
    \[ \|\Roff\|_{L^1} \le \sum_{\substack{k,l \in \Lambda\\k \ne l}} \|a_k\|_{L^\infty} \|a_l\|_{L^\infty} \|W_k(\gamma x) \otimes W_l(\gamma x)\|_{L^1} \le C_u \mu^{-1}. \]

    The estimate \eqref{e:Rcor-est} directly follows from H\"older's inequality.

    For \eqref{e:Rlin-est}, note that $\cR \Div$ is a Calder\'on-Zygmund operator, so we have that for any $1<p<\infty$,
    \begin{align*}
        \|\cR\Div(u \otimes w+w \otimes u)\|_{L^1} 
        &\le \|\cR\Div(u \otimes w+w \otimes u)\|_{L^p} \\
        &\lesssim \|u\|_{L^\infty} \|w\|_{L^p} \le C_u \|w\|_{L^p}.
    \end{align*}
    Moreover, as $-s+2\alpha-1-d/2 < -d$, Sobolev's embedding theorem yields that $L^1$ embeds into $H^{-s+2\alpha-1}$, so we get
    \[ \|\cR(-\Delta)^\alpha w\|_{H^{-s}} \lesssim \|w\|_{H^{-s+2\alpha-1}} \lesssim \|w\|_{L^1} \le \|w\|_{L^p}. \]
    Combining these estimates gives \eqref{e:Rlin-est}. This completes the proof.
\end{proof}

\begin{cor}[Estimates on paraproducts]
    \label{cor:para-est}
    Let $s \ge \frac{d}{2}+2\alpha$. There exists $C_3 \ge 1$ depending only on $(u,R)$ with the following estimates
    \[ \|w \otimes w\|_{\DotH^{-s}} \le \|R\|_{H^{-s}} + C_3\left(\|\Rosc\|_{H^{-s}} + \|\Rdis\|_{H^{-s}} + \|\Roff\|_{L^1} + \|\Rcor\|_{L^1}\right), \]
    and
    \[ \sum_{M \le 2N} \left( \|P_M u \otimes w\|_{\DotH^{-s}} + \|w \otimes P_M u\|_{\DotH^{-s}} \right) \le C_3 \|w\|_{L^p}. \]
\end{cor}

\begin{proof}
    Using again $L^1 \hookrightarrow H^{-s}$, we get from \eqref{e:R+w*w} that
    \begin{align*}
        \|w \otimes w\|_{\DotH^{-s}} 
        &\le \|R\|_{H^{-s}} + \|w \otimes w+R\|_{\DotH^{-s}} \\
        &\le \|R\|_{H^{-s}} + C_u \left( \|\Rosc\|_{H^{-s}} + \|\Rdis\|_{H^{-s}} + \|\Roff\|_{L^1} + \|\Rcor\|_{L^1} \right).
    \end{align*}
    Moreover, H\"older's inequality yields that
    \[ \sum_{M \le 2N} \|P_M u \otimes w\|_{\DotH^{-s}} \lesssim \sum_{M \le 2N} \|P_M u \otimes w\|_{L^1} \lesssim_N \|u\|_{L^\infty} \|w\|_{L^p}, \]
    and the other term follows similarly. This completes the proof.
\end{proof}
%%%%% Pf Proposition
\section{Proof of Proposition \ref{prop:iteration}}
\label{sec:pf-prop}

In this section, we prove Proposition \ref{prop:iteration} via a particular choice of the parameters. Let $\lambda \ge 2$. We first determine $\sigma$, $\mu$, and $\gamma$ with respect to $\lambda$. The value of $\sigma$ is given by Lemma \ref{lemma:arithmetic-sigma}. Also recall that
\[\theta \in (0,1), \quad e:=2\left \lceil \theta^{-1} \right \rceil > 2.  \]
Define two exponents $\epsilon$ and $\beta$ by
\[ \epsilon := \frac{2-p}{12}, \quad \beta := \frac{p}{6(d-1)} = \frac{\epsilon}{(d-1)(\frac{1}{p}-\frac{1}{2})}. \]
Set
\begin{equation}
    \label{e:mu-gamma}
    \mu := \lambda^\beta, \quad \gamma := \lambda^{1/2}.
\end{equation}

\begin{lemma}
    \label{lemma:parameters}
    Let $d \ge 2$ and $p \in (1,2)$ with $\frac{1}{p}-\frac{1}{2} \le \frac{1}{d-1}$. Then
    \begin{align}
        \label{e:mu-1-est}
        \gamma^{-1} \le \mu^{-1} &\le \lambda^{-\epsilon}, \\
        \label{e:mu(d-1)[2,p]-est}
        \mu^{(d-1)(\frac{1}{2}-\frac{1}{p})} &\le \lambda^{-\epsilon}, \\
        \label{e:lambda-1*gamma*mu-est}
        \lambda^{-1} \gamma \mu \le \lambda^{-1} \gamma \mu^{1+\frac{1}{2}(d-1)} &\le \lambda^{-\epsilon}.
    \end{align}
    Moreover, for $s \ge 1$, it holds that
    \begin{align}
        \label{e:gamma-s*mu-est}
        \gamma^{-s} \mu^{\frac{1}{2}(d-1)} &\le \lambda^{-\epsilon}, \\
        \label{e:sigma-1*mu-est}
        \sigma^{-s} \mu^{\frac{1}{2}(d-1)} \le \sigma^{-1} \mu^{\frac{1}{2}(d-1)} \le \sigma^{-\theta} \mu^{\frac{1}{2}(d-1)} &\le \lambda^{-\epsilon}.
    \end{align}
\end{lemma}

\begin{proof}
    First, \eqref{e:mu-1-est} is equivalent to $\epsilon \le \beta \le 1/2$. The first inequality holds as $(d-1)(\frac{1}{p}-\frac{1}{2}) \le 1$, and the second holds as $p<2$ and $d \ge 2$.
    
    Next, \eqref{e:mu(d-1)[2,p]-est} follows by definition as $\mu^{(d-1)(\frac{1}{2}-\frac{1}{p})} = \lambda^{-\epsilon}$.

    To prove \eqref{e:lambda-1*gamma*mu-est}, direct computation shows that
    \[ \lambda^{-\frac{1}{6}} \mu^{\frac{1}{2}(d-1)} = \lambda^{-\frac{1}{6}+\frac{1}{2}(d-1)\beta} = \lambda^{-\frac{1}{6}+\frac{p}{12}} = \lambda^{-\epsilon}. \]
    Then using $\gamma,\mu \ge 1$ and $d \ge 2$, we obtain \eqref{e:lambda-1*gamma*mu-est} as
    \[ \lambda^{-1} \gamma \mu \le \lambda^{-1} \gamma \mu^{1+\frac{1}{2}(d-1)} \le \lambda^{-\frac{1}{3}} \mu \lambda^{-\frac{1}{6}} \mu^{\frac{1}{2}(d-1)} \le \left( \lambda^{-\frac{1}{6}} \mu^{\frac{1}{2}(d-1)} \right)^3 \le \lambda^{-\epsilon}. \]

    The estimates \eqref{e:gamma-s*mu-est} and \eqref{e:sigma-1*mu-est} directly follow from \eqref{e:lambda-1*gamma*mu-est} as $\gamma^{-s} \le \gamma^{-1} = \lambda^{-1/2} \le \lambda^{-1} \gamma \mu$ and $\sigma^{-s} \le \sigma^{-1} < \sigma^{-\theta} < \lambda^{-e\theta} < \lambda^{-1} \le \lambda^{-1} \gamma \mu$. 
\end{proof}

Let us present the proof of Proposition \ref{prop:iteration}.

\begin{proof}[Proof of Proposition \ref{prop:iteration}]
    Let $d \ge 2$ and $\alpha>0$. Set
    \[ s = d/2+2\alpha+1. \] 

    Let us verify the properties. Without loss of generality, we prove the statements for $p$ sufficiently close to 2 so that
    \[ \frac{1}{p}-\frac{1}{2} \le \frac{1}{d-1}. \]
    The other cases follow by H\"older's inequality. Let $0<\theta<1$, $\delta>0$, and $N$ be a dyadic integer. Let $(u,R)$ be the given smooth solution to \eqref{e:NSR} with $\supp \hatu \subset B(0,N)$. Write $e:=2\left\lceil \theta^{-1} \right\rceil$.
    
    Let $\lambda \ge 2$ be a sufficiently large dyadic number (to be determined) with $\gamma=\lambda^{1/2} \in \bN$. Pick the dyadic integer $\sigma>50\lambda^e$ as in Lemma \ref{lemma:arithmetic-sigma} and the parameter $\mu$ as in \eqref{e:mu-gamma}. 
    
    Applying these parameters to Lemma \ref{lemma:ovu,ovR-sol} gives a new smooth solution $(\ovu,\ovR)$ such that the Fourier transform of the velocity perturbation $w:=\ovu-u$ is supported in the annulus
    \[ \left\{ m \in \bZ^d:\frac{1}{2}\sigma < |m| < \frac{9}{10}\sigma \right\}. \]
    Moreover, we infer from Propositions \ref{prop:w-est}, \ref{prop:R-est}, Corollary \ref{cor:para-est}, and Lemma \ref{lemma:parameters} that there exists a constant $C>0$ depending only on $(u,R)$ with the following estimates
    \begin{align*}
        \|\ovR\|_{H^{-s}} &\le C \lambda^{-\epsilon}, \\
        \|w\|_{L^p} &\le C \lambda^{-\epsilon}, \\
        \|w\|_{B^{-\theta}_{\infty,1}} &\le C \lambda^{-\epsilon}, \\
        \|w \otimes w\|_{\DotH^{-s}} + \sum_{M \le 2N} \big(\|P_M u \otimes w\|_{\DotH^{-s}} + \|w \otimes P_M u\|_{\DotH^{-s}}\big) &\le \|R\|_{H^{-s}} + C\lambda^{-\epsilon}.
    \end{align*}

    Finally, pick $\lambda$ as a dyadic integer so that
    \[ \lambda > \max\{ C^{1/\epsilon} \delta^{-1/\epsilon}, 4N \} ~ \text{ and } ~ \lambda^{1/2} \in \bN.  \]
    In particular, one has $C\lambda^{-\epsilon}<\delta$. The estimates \eqref{e:ovR-est}, \ref{item:prop-w-Lp}, and \ref{item:prop-w-paraproduct} follow from the inequalities right above. Moreover, we get \ref{item:prop-hatw-supp} by taking
    \[ D := \sigma/N > 50\lambda^e/N > 50\lambda^2/N > 50 N. \]
    This completes the proof.
\end{proof}

We explain here why our non-trivial stationary solution $u$ to \eqref{e:NSE} does not belong to the Chemin-Lerner space $\tilL^p_t B^{-1+\frac{d}{q}+\frac{2}{p}}_{q,r}$ for $d<q<\infty$, $1 \le r< \infty$, and $2<p<\frac{2q}{q-d}$. Indeed, observe that for any finite time $T>0$, any stationary solution $u$ satisfies that
\[ \|u\|_{\tilL^p_t B^{-1+\frac{d}{q}+\frac{2}{p}}_{q,r}} = \left( \sum_N N^{(-1+\frac{d}{q}+\frac{2}{p})r} \|P_N u\|_{L^p_t L^q_x}^r \right)^{1/r} \eqsim T^{1/p} \|u\|_{B^{-1+\frac{d}{q}+\frac{2}{p}}_{q,r}}. \]
The conditions on $q$ and $p$ ensure that $-1+\frac{d}{q}+\frac{2}{p}>0$, so $B^{-1+\frac{d}{q}+\frac{2}{p}}_{q,r}$ embeds into $L^q$. Corollary \ref{cor:sing-sol-Lp=0} says that there does not exist non-trivial stationary solutions in $L^q$. Particularly, our stationary singular solution does not belong to either $B^{-1+\frac{d}{q}+\frac{2}{p}}_{q,r}$ or $\tilL^p_t B^{-1+\frac{d}{q}+\frac{2}{p}}_{q,r}$.

However, on the other hand, Theorem \ref{thm:main} proves existence of singular stationary solutions in $\tilL^p_t B^{-1+\frac{d}{q}+\frac{2}{p}}_{q,r}$ for all $\frac{2q}{q-d} < p \le \infty$.
%%%%% Pf Proposition L2
\section{Proof of Proposition \ref{prop:L2-iteration}}
\label{sec:pf-L2-iteration}

The proof resembles the proof of Proposition \ref{prop:iteration}.
We introduce a cut-off function $\zeta:\mat_d(\bR) \to [0,\infty)$ which is increasing with respect to the matrix norm $|\cdot|$ and satisfies
\[ \zeta(X) := \begin{cases}
4\|R\|_{L^1}  & \text{ if } |X| \le \|R\|_{L^1} \\
4|X|  & \text{ if } |X| \ge 2\|R\|_{L^1}.
\end{cases} \]
Apply it to the Reynolds stress and define $\rho:\bT^d \to [0,\infty)$ by
\[ \rho(x) := \zeta(R(x)). \]
It is clear that $\bI-\frac{R}{\rho}$ lies in $B_{1/2}(\bI)$.

Define the new principal part $\wprin$ by
\[ \wprin(x) := \sum_{k \in \Lambda} a_k(x) W_k(\gamma x), \]
where the amplitude function $a_k:\bT^d \to \bR$ is given by
\[ a_k := \rho^{1/2} \Gamma_k \left( \bI - \frac{R}{\rho} \right). \]
We do not need the frequency localization corrector in this case. Namely, the new corrector $\wcor$ is given by
\[ \wcor(x) := \gamma^{-1} \sum_{k \in \Lambda} \nabla a_k(x) : \Omega_k(\gamma x). \]
The velocity perturbation $w=\ovu-u$ is hence given by
\[ w := \wprin + \wcor, \]
which is divergence-free.

Define the linear error $\Rlin$ by
\[ \Rlin := \cR(-\Delta)^\alpha w + u \otimes w + w \otimes u, \]
the correction error $\Rcor$ by
\[ \Rcor := \wprin \otimes \wcor + \wcor \otimes w, \]
and the off-diagonal error $\Roff$ by
\[ \Roff := \sum_{k \ne l} a_k a_l W_k(\gamma x) \otimes W_l(\gamma x). \]
To define the oscillation error, we need to introduce a bilinear anti-divergence operator $\cB$ from $C^\infty(\bT^d;\bR^d) \times C^\infty(\bT^d;\mat_d(\bR))$ to $C^\infty(\bT^d;\mat_d(\bR))$ defined by
\[ (\cB(f,X))_{i,j} = f_l \cR_{i,j,k} X_{l,k} - \cR(\partial_i f_l \cR_{i,j,k} X_{l,k}). \]
It is clear that $\cB(f,X)$ is a traceless symmetric matrix. The reader can refer to \cite[Appendix B.3]{Cheskidov-Luo2022-LpLinfty} for more detailed properties of $\cB$. Define the oscillation error $\Rosc$ by
\[ \Rosc := \sum_{k \in \Lambda} \cB( \nabla(a_k^2),T_k), \]
where $T_k$ is the oscillation part of $W_k(\gamma x) \otimes W_k(\gamma x)$ as defined in \eqref{e:Tk}. The new Reynolds stress $\ovR$ is hence defined by
\[ \ovR := \Rosc + \Roff + \Rcor + \Rlin. \]

Observe that $(\ovu,\ovR)$ is a smooth solution to \eqref{e:NSR}. Indeed, direct computation shows that
\[ \Div \ovR = \Div(R+w \otimes w) + (-\Delta)^\alpha w + \Div(u \otimes w + w \otimes u) - \nabla \rho, \]
which agrees with \eqref{e:ovR-eq} with $P=-\rho$.

Now, we summarize the estimates needed.
\begin{lemma}[Estimates on $w$]
    \label{lemma:w-est-L2}
    There exists $A>0$ depending only on $d$ and $C_1>0$ depending on $(u,R)$ such that for any $\beta \ge 0$ and $1 \le p \le \infty$, it holds that
    \begin{align}
        \label{e:wp-est-L2}
        \|\wprin\|_{L^2} &\le A \|R\|_{L^1}^{1/2} + C_1 \gamma^{-1/2}, \\
        \label{e:wc-est-L2}
        \|\wcor\|_{L^p} &\le C_1 \gamma^{-1} \mu^{-1+(d-1)(\frac{1}{2}-\frac{1}{p})}, \\
        \label{e:w-est-L2-Wb,p}
        \|(-\Delta)^\beta w\|_{L^p} & \le C_1 \gamma^{2\beta} \mu^{2\beta+(d-1)(\frac{1}{2}-\frac{1}{p})}, \\
        \label{e:PNw-est-L2}
        \|P_{\le \gamma/2} w\|_{L^2} &\lesssim C_1 \gamma^{-1}.
    \end{align}
\end{lemma}

\begin{proof}
    To prove \eqref{e:wp-est-L2}, we use the improved H\"older's inequality (see \cite[Lemma 2.1]{Modena-Szekelyhidi2018} and \cite[Lemma 3.7]{Buckmaster-Vicol2019-C0L2}) to get that
    \[ \|\wprin\|_{L^2} \lesssim \sum_{k \in \Lambda} \left(\|a_k\|_{L^2} \|W_k\|_{L^2} + \gamma^{-1/2} \|a_k\|_{C^1} \|W_k\|_{L^2}\right). \]
    Since $\Lambda$ and $\Gamma_k$ only depend on $d$, we have
    \[ \|a_k\|_{L^2} \lesssim_d \|\rho\|_{L^1}^{1/2} \lesssim \|R\|_{L^1}^{1/2}. \]
    Gathering these estimates gives \eqref{e:wp-est-L2}.

    Next, \eqref{e:wc-est-L2} follows from Lemma \ref{lemma:W_k} as
    \[ \|\wcor\|_{L^p} \le \gamma^{-1} \sum_{k \in \Lambda} \|\nabla a_k\|_{L^\infty} \|\Omega_k\|_{L^p} \lesssim C_u \gamma^{-1} \mu^{-1+(d-1)(\frac{1}{2}-\frac{1}{p})}. \]

    For \eqref{e:w-est-L2-Wb,p}, since $W_k$ is of mean zero, we infer from Lemma \ref{lemma:W_k} that
    \begin{align*}
        \|(-\Delta)^\beta \wprin\|_{L^p} 
        &\lesssim \sum_{k \in \Lambda} \|a_k\|_{C^{2\lceil \beta \rceil+1}} \|(-\Delta)^\beta (W_k(\gamma x))\|_{L^p} \\
        &\le C_u \gamma^{2\beta} \|(-\Delta)^\beta W_k\|_{L^p} \le C_u \gamma^{2\beta} \mu^{2\beta+(d-1)(\frac{1}{2}-\frac{1}{p})}.
    \end{align*}
    Similar computation shows that
    \[ \|(-\Delta)^\beta \wcor\|_{L^p} \lesssim \gamma^{2\beta-1} \mu^{2\beta-1+(d-1)(\frac{1}{2}-\frac{1}{p})}. \]
    As $\gamma \ge 1$ and $\mu \ge 1$, combining these estimates gives \eqref{e:w-est-L2-Wb,p}.

    Finally, to prove \eqref{e:PNw-est-L2}, note that as $W_k$ is of mean zero, the Fourier support of $W_k(\gamma \cdot)$ is contained in $\{m \in \bZ^d:|m| \ge \gamma\}$. So we get
    \begin{align*}
        \|P_{\le \gamma/2} w\|_{L^2}
        &\le \|P_{\le \gamma/2} \wcor\|_{L^2} + \|P_{\le \gamma/2} \wprin\|_{L^2} \\
        &\le C_u \gamma^{-1} \mu^{-1} + \sum_{k \in \Lambda} \|P_{>\gamma/2} a_k\|_{L^\infty} \|W_k\|_{L^2} \le C_u \gamma^{-1}.
    \end{align*}
    This completes the proof.
\end{proof}

\begin{lemma}[Estimates on $\ovR$]
    \label{lemma:R-est-L2}
    Let $1<r<\infty$. There exists $C_2>0$ depending on $(u,R)$ and $r$ with the following estimates
    \begin{align}
        \label{e:Rosc-est-L1}
        \|\Rosc\|_{L^1} &\le C_2 \gamma^{-1}, \\
        \label{e:Roff-est-L1}
        \|\Roff\|_{L^1} &\le C_2 \mu^{-1}, \\
        \label{e:Rcor-est-L1}
        \|\Rcor\|_{L^1} &\le C_2 \gamma^{-1} \mu^{-1}, \\
        \label{e:Rlin-est-L1}
        \|\Rlin\|_{L^1} &\le C_2 \left( \mu^{-\frac{1}{2}(d-1)} + \gamma^{2\alpha-1} \mu^{2\alpha-1+(d-1)(\frac{1}{2}-\frac{1}{r})} \right).
    \end{align}
\end{lemma}

\begin{proof}
    First consider \eqref{e:Rosc-est-L1}. Notice that $\hatT_k$ is supported in $\{m \in \bZ^d:|m| \ge \gamma\}$. So \eqref{e:Rosc-est-L1} follows from the bilinear estimates of $\cB$ (see \cite[Theorem B.4]{Cheskidov-Luo2022-LpLinfty}) and Bernstein's inequality as
    \[ \|\Rosc\|_{L^1} \lesssim \sum_{k \in \Lambda} \|\nabla (a_k^2)\|_{C^1} \|\cR T_k\|_{L^1} \lesssim \sum_{k \in \Lambda} \|a_k\|_{C^2}^2 \gamma^{-1} \|T_k\|_{L^1} \le C_u \gamma^{-1}. \]

    The proof of \eqref{e:Roff-est-L1} is the same as that of \eqref{e:Roff-est} using Lemma \ref{lemma:W_k}.

    The estimate \eqref{e:Rcor-est-L1} follows from H\"older's inequality and Lemma \ref{lemma:w-est-L2}.

    To prove \eqref{e:Rlin-est-L1}, we use Lemma \ref{lemma:w-est-L2} to get
    \[ \|u \otimes w+w \otimes u\|_{L^1} \le C_u \|w\|_{L^1} \le C_u \mu^{-\frac{1}{2}(d-1)}.  \]
    For $\alpha<1/2$, we infer from Bernstein's inequality and Lemma \ref{lemma:w-est-L2} that
    \[ \|\cR(-\Delta)^\alpha w\|_{L^1} \lesssim \|w\|_{L^1} \lesssim \mu^{-\frac{1}{2}(d-1)}. \]
    For $\alpha \ge 1/2$, since $\cR (-\Delta)^{1/2}$ is a Calder\'on-Zygmund operator, we get
    \begin{align*}
        \|\cR(-\Delta)^\alpha w\|_{L^1} 
        &\le \|\cR(-\Delta)^{1/2} (-\Delta)^{\alpha-1/2} w\|_{L^r} \\
        &\lesssim \|(-\Delta)^{\alpha-1/2} w\|_{L^r} \le C_u \gamma^{2\alpha-1} \mu^{2\alpha-1+(d-1)(\frac{1}{2}-\frac{1}{r})}.
    \end{align*}
    Gathering these estimates gives \eqref{e:Rlin-est-L1} as desired.
\end{proof}

Now, let us prove Proposition \ref{prop:L2-iteration}.
\begin{proof}[Proof of Proposition \ref{prop:L2-iteration}]
    Let $\gamma \ge 1$ be a dyadic number to be determined. We first determine $\mu$ with respect to $\gamma$. As $\alpha < (d+1)/4$, one has
    \[ \epsilon := \frac{1}{3} \min\left\{ 1,\frac{1}{2}(d-1)-(2\alpha-1) \right\} > 0. \]
    Define
    \[ \mu := \gamma^{2\lceil \alpha/\epsilon \rceil} \geq \gamma^2. \]
    By perturbation, pick $r \in (1,2)$ sufficiently close to 1 such that
    \[ (d-1)\left( \frac{1}{r}-\frac{1}{2} \right) - (2\alpha-1) \ge 2\epsilon. \]
    Then one has
    \[ \gamma^{2\alpha-1} \mu^{2\alpha-1+(d-1)(\frac{1}{2}-\frac{1}{r})} \le \gamma^{2\alpha-1} \mu^{-2\epsilon} \le \gamma^{-1-2\alpha} \le \gamma^{-1}, \]
    and
    \[ \mu^{-\frac{1}{2}(d-1)} \le \mu^{-1/2} \le \gamma^{-1}. \]
    Therefore, Lemmas \ref{lemma:w-est-L2} and \ref{lemma:R-est-L2} yield that there exists $A>0$ depending only on $d$ and $C>0$ depending on $(u,R)$, such that the following estimates hold:
    \begin{align*}
        \|w\|_{L^2} &\le A\|R\|_{L^1}^{1/2} + C\gamma^{-1/2}, \\
        \|P_{\le \gamma/2} w\|_{L^2} &\le C\gamma^{-1}, \\
        \|\ovR\|_{L^1} &\le C\gamma^{-1}.
    \end{align*}
    Then Proposition \ref{prop:L2-iteration} follows from a particular choice of $\gamma$ as a dyadic number such that
    \[ \gamma>10N, \quad C\gamma^{-1/2}<\delta, \quad\text{and} ~ C\gamma^{-1}<\min\{\delta^2,\delta'\}. \]
    This completes the proof.
\end{proof}

\section*{Statements and declarations}
\begin{itemize}
    \item Conflict of interest: On behalf of all authors, the corresponding author states that the authors have no financial or non-financial interest to disclose, and there is no conflict of interest.
    \item Data availability: Data sharing not applicable to this article as no datasets were generated or analyzed during the current study.
    \item Open access: This article is licensed under a Creative Commons Attribution 4.0 International License, which permits use, sharing, adaptation, distribution and reproduction in any medium or format, as long as you give appropriate credit to the original author(s) and the source, provide a link to the Creative Commons license, and indicate if changes were made. The images or other third party material in this article are included in the article’s Creative Commons license, unless indicated otherwise in a credit line to the material. If material is not included in the article’s Creative Commons license and your intended use is not permitted by statutory regulation or exceeds the permitted use, you will need to obtain permission directly from the copyright holder. To view a copy of this license, visit \url{https://creativecommons.org/licenses/by/4.0/}.
\end{itemize}

\bibliographystyle{alpha}
\bibliography{refs}

\end{document}